\documentclass[12pt]{amsart}
\usepackage{graphicx}
\usepackage{amssymb}
\usepackage{amsmath}
\usepackage{amsthm}
\usepackage{amscd}
\usepackage{color}
\usepackage{colordvi}
\usepackage[all,2cell]{xy}
\usepackage{tikz}
\usetikzlibrary{matrix, arrows}
\UseAllTwocells \SilentMatrices
\newtheorem{thm}{Theorem}[section]

\newtheorem{cor}[thm]{Corollary}
\newtheorem{lem}[thm]{Lemma}

\newtheorem{prop}[thm]{Proposition}
\theoremstyle{definition}
\theoremstyle{remark}
\newtheorem{rem}[thm]{\bf Remark}
\numberwithin{equation}{section}

\newtheorem*{thm*}{\bf Theorem A}
\newtheorem*{thm**}{\bf Theorem B}
\newtheorem*{thm***}{\bf Theorem C}
\newtheorem{exa}[thm]{\bf Example}
\theoremstyle{definition}
\newtheorem{df}[thm]{Definition}

\newtheorem*{ques*}{Question}
\numberwithin{equation}{section}
\newcommand{\Ext}{\mbox{\rm Ext}}
\newcommand{\Hom}{\mbox{\rm Hom}}

\newcommand{\Z}{\mbox{\rm Z}}

\newcommand{\Ker}{\mbox{\rm Ker}}
\newcommand{\rep}{\mbox{\rm Rep}}

\newtheorem{nota}[thm]{Notation}

\begin{document}
\title[Homological theory of representations]{Homological theory of representations having pure acyclic injective resolutions}

\dedicatory{}%
\commby{}%
\vspace{-5mm}
\author[Qihui Li, Junpeng Wang, Gang Yang] {Qihui Li, Junpeng Wang, Gang Yang$^*$ }

\thanks{\emph{2020 Mathematics Subject Classification}: 16G20,  18A40, 18G05, 18G20.}

\thanks{$^*$ Corresponding author.}

\keywords{Injective representation, strongly fp-injective representation, Gorenstein injective representation, model structure.}%

\maketitle

\vspace{-10mm}

\begin{abstract}
Let $Q$ be a quiver and $R$ an associative ring. A representation by $R$-modules of $Q$ is called
strongly fp-injective if it admits a pure acyclic injective resolution in
the category of representations.
It is shown that such representations possess many nice properties.
We characterize  strongly fp-injective representations under some mild  assumptions, which  is
closely related to strongly fp-injective $R$-modules.
Subsequently, we use such representations  to define relative Gorenstein injective representations,
called Gorenstein strongly fp-injective representations,
and give an explicit characterization of  the  Gorenstein strongly fp-injective representations of {certain} right rooted quivers.
As an application, a model structure in  the category of representations is given.
\end{abstract}

\section{Introduction}
The study of relation between  representations  and their components  is a traditional and  important
research subject in the theory of homological algebra. The point of this paper aims at studying how the homological theory on the
category $R$-Mod of modules related to strongly fp-injective and Gorenstein strongly fp-injective
modules can be extended to a homological theory on the category  of representations.
To introduce the research work that inspired us, we first introduce some needed notations.
Throughout the paper,  let $R$ be an associative ring with
identity, and $Q=(Q_0, Q_1)$ a quiver, where $Q_0$ and $Q_1$ are the sets of  all vertices and arrows in $Q$, respectively.
Denote by $\text{Rep}(Q, R)$ the category of representations by $R$-modules of $Q$.  For any
object $X$ in $\text{Rep}(Q, R)$ and any vertex $i\in Q_0$, there are two induced
canonical homomorphisms
$$\psi_i^X: X(i)\to\prod_{a: i\rightarrow j}X(j) \hspace{0.5cm} \text{and} \hspace{0.5cm} \varphi_i^X: \displaystyle\bigoplus_{a: j\rightarrow i}X(j)\to X(i),$$ where $a\in Q_1$ is the arrow.

Enochs, Estrada and  Garc\'{\i}a Rozas gave in \cite[Theorem 4.2]{EEGR09} and in \cite[Theorem 3.1]{EE05} the nice classification for injective and projective representations over right and left rooted quivers, respectively.  They showed that a representation $X$ in $\text{Rep}(Q, R)$ of a right rooted  (respectively,
left rooted) quiver $Q$  is injective (respectively, projective) if and only if  the canonical homomorphism $\psi_i^X$ (respectively,
$\varphi_i^X$) is a splitting epimorphism (respectively,
splitting monomorphism), and the left $R$-module $X(i)$ is  injective (respectively,
projective)  for each vertex  $i\in Q_0$. Analogously, a characterization of flat representations
 (colimits of projective representations) has been given by Enochs, Oyonarte and
Torrecillas in \cite[Theorem 3.7]{EOT04}, that is, a representation $X$ in $\text{Rep}(Q, R)$ of a left rooted  quiver $Q$  is flat if and only if  the canonical homomorphism $\varphi_i^X$ is a pure monomorphism, and the left $R$-module $X(i)$ is  flat for each vertex  $i\in Q_0$.

According to
\cite{Sten70}, an $R$-module $L$ is said to be fp-injective if $\Ext^1_R(P, L)=0$ for every finitely presented $R$-module $P$, or equivalently, if every exact sequence
$0\to L\to M\to N\to 0$ of $R$-modules is pure. In this sense, fp-injective modules are also known as absolutely pure modules. We notice that an $R$-module $N$ is flat if and only if every exact sequence $0\to L\to M\to N\to 0$ of $R$-modules is pure, and so fp-injective modules are  often regarded as dual analogues of flat modules. It is well-known that the class of all flat $R$-modules is closed under  taking kernels of epimorphisms over any associative ring.  However, the class of all fp-injective $R$-modules isn't closed under  taking cokernels of monomorphisms over an associative ring in general, in fact, it is  closed under taking cokernels of monomorphisms if and only if $R$ is left coherent. To remedy the failure of the class of all fp-injective $R$-modules to be closed under  taking cokernels of monomorphisms  without the coherent condition for the ring, Emmanouil and Kaperonis \cite{EK}, and  Li, Guan and Ouyang \cite{LGO}  paid special attention to the narrower class of strongly fp-injective modules, where an $R$-module $L$ is called \emph{strongly fp-injective} if $\Ext^n_R(P, L)=0$ for any finitely presented $R$-module $P$ and all $n\geq 1$ (\cite[Definition 3.1]{LGO}).

We are inspired in this paper to investigate strongly fp-injective representations in the category $\text{Rep}(Q, R)$, and we will show that such representations share many nice homological properties.
We use the pure exact sequence in the category of representations of quivers to give its definition. Specifically, we call a representation $X$ \emph{strongly fp-injective} if it admits a pure acyclic injective resolution  in the category $\text{Rep}(Q, R)$. This program of the definition ensures that the class of all strongly fp-injective representations behaves nicely without the coherent condition for the ring.  In section 3, we succeed in getting the following result:
\begin{thm*} (Propositions \ref{closeprop}, \ref{closeprop1}, \ref{closed p1}, \ref{Stability}) The following statements hold for any quiver $Q$:
   \begin{enumerate}
     \item The class of all strongly fp-injective representations is closed under
extensions, finite direct sums and direct summands.
     \item The class of all strongly fp-injective representations is closed under
taking cokernels of monomorphisms.
    \item Let  $0\rightarrow X\rightarrow Y\rightarrow Z\rightarrow 0$ be a pure exact sequence of  representations.
If $Y$ and $Z$ are strongly fp-injective, then so is $X$.
     \item A representation $X$ is strongly fp-injective if and only if there exists a pure exact sequence $0 \to X \to T^0
\to T^1 \to \cdots$ with $T^i$ strongly fp-injective in $\text{Rep}(Q,R)$ for each integer $i\geq 0$.
   \end{enumerate}
\end{thm*}

It is worth to point out that our proof of the above result is very  ingenious since we use only the purity and the characterization of injective representations of quivers instead of computing the  homology. In this way, we even don't care about what does a finitely presented representation in $\text{Rep}(Q, R)$ look like. However, in the category of $R$-modules, to show that an $R$-module $M$ is strongly fp-injective, one usually needs to prove that the homology $\Ext^i_R(N,M)$ vanishes for any finitely presented $R$-module $N$ and any integer $i\geq 1$.

To clarify the structure of  strongly fp-injective representations, we prove that such representations are closely connected to strongly fp-injective modules as the following:

\begin{thm**} (Theorems \ref{strongly FP},  \ref{strongly FP right}) Let $Q$ be a  {locally target-finite} quiver and $X$
a  representation in $\text{Rep}(Q, R)$.   If  $X$ is strongly fp-injective, then  the $R$-module $X(i)$ is strongly fp-injective, and
    the canonical homomorphism $\psi_i^X$  is a pure epimorphism for each vertex  $i\in Q_0$. The converse also holds
    provided that $Q$ is right rooted.
 \end{thm**}

In section 4, we consider the  relative Gorenstein injective representations and introduce  the  Gorenstein strongly fp-injective representations. We
give an explicit characterization of  the  Gorenstein strongly fp-injective representations of certain right rooted quivers. More
precisely, we obtain the following result.

\begin{thm***} (Theorem \ref{main Gorens}) Let $Q$ be a   {locally target-finite and} right rooted quiver. Then a representation $X$ in $\text{Rep}(Q, R)$ is Gorenstein strongly fp-injective  if and only if the $R$-modules $X(i)$ and $\Ker(\psi_i^X)$ are  Gorenstein strongly fp-injective, and the canonical morphism $\psi_i^X$  is  epimorphic for each vertex  $i\in Q_0$.
\end{thm***}

It is known that a ring $R$ is left coherent if and only if any  fp-injective $R$-module is  strongly fp-injective. Thus it is particularly interesting to note that the above result reveals incidently the behavior of the Ding injective representations of $R$-modules over left coherent rings, which is given in Corollary \ref{Ding injective}, and then this helps us to improve the model structure in \cite[Proposition 5.10]{BGH}  from the category $R$-Mod to the category $\text{Rep}(Q, R)$  when $R$ is restricted to a left coherent ring, see Corollary \ref{Ding injective Model}.

\section{preliminaries}

Throughout the paper $R$ denotes an associative ring with  identity. An $R$-module will
mean a left $R$-module, unless stated otherwise. We also refer to right $R$-modules as modules over the opposite ring $R^{op}$.
The category of $R$-modules will be denoted $R$-Mod. In this section, we recall some notions and facts which will be used throughout the paper.

\subsection*{Quiver representations}
A quiver, denoted by $Q=(Q_0, Q_1)$, is a directed graph, where $Q_0$ and $Q_1$ are the set of vertices and arrows, respectively.
For $i, j\in Q_0$ (not necessarily different), $Q(i, j)$ denotes the set of paths in $Q$ from $i$ to $j$.
For every $i\in Q_0$, $e_i$ is the trivial path. If $p\in Q(i, j)$, we write $s(p)=i$ and $t(p)=j$, called its source and target,
respectively. Thus, one can think of $Q $ as a category with the object class $Q_0$ and $\Hom_Q(i, j):=Q(i, j)$.

Let $Q$ be a quiver, and $R$ a ring. A representation $X$ of $Q$ by $R$-modules is  a covariant functor $X: Q\to R\text{-Mod}$.
Thus a representation $X$ is determined by assigning an $R$-module
 $X(i)$ to each vertex $i\in Q_0$ and an $R$-homomorphism $X(a): X(i)\to X(j)$ to each arrow $a\in Q_1$.
 A morphism $f: X\to Y$ between representations $X$ and $Y$ is
a natural transformation, that is, a family of  homomorphisms $\{f(i): X(i)\to Y(i)\}_{i\in Q_0}$ such that
$Y(a)\circ f(i)=f(j)\circ X(a)$, i.e.,  the following diagram is commutative
$$\xymatrix{X(i)\ar[d]_{f(i)} \ar[r]^{X(a)} & X(j) \ar[d]^{f(j)} \\
  Y(i) \ar[r]^{Y(a)} & Y(j)   }$$  for each arrow $a: i\to j$ in $Q_1$. A
representation $P$ of $Q$ is said to be finitely presented if the functor  $\Hom_Q(P, -)$
presevers direct limits.  The category of representations of a quiver $Q$ by
$R$-modules over a ring $R$ is denoted by $\text{Rep}(Q, R)$. This is a locally finitely presented additive
category, see \cite{AR94, CB94}, actually it is a Grothendieck
category with enough projectives  and injectives. It is known that if  $Q$ is finite, i.e., $Q_0$ and $Q_1$ are both finite, then the category $\text{Rep}(Q, R)$ is
equivalent to the category of $RQ$-modules, where $RQ$ is the path ring of $Q$. For the
two representations $X$ and $Y$ of $Q$, we write $\Hom_Q(X,Y)$ for the abelian group of
morphisms from $X$ to $Y$ in the category $\text{Rep}(Q, R)$.

\subsection*{Right adjoint of the restriction functor}
Let $Q$ be a quiver and $Q'=(Q'_0, Q'_1)$ be its subquiver. According to \cite{AEHS11}, the \emph{restriction functor}
$$e^{Q'}: \text{Rep}(Q, R)\longrightarrow\text{Rep}(Q', R)$$  which restricts any representation of
$Q$ to the vertices of $Q'$ is known to possess the \emph{right adjoint}
$$e^{Q'}_{\rho}: \text{Rep}(Q', R)\longrightarrow\text{Rep}(Q, R)$$
given as the following description  for any representation $X\in\text{Rep}(Q', R)$, if $v \in Q'_0$,
then
$$e^{Q'}_{\rho}(X)(v)=X(v)\bigoplus\displaystyle\prod_{\begin{array}{c}s(a\alpha)=v \\
                                                                  t(a\alpha)\in Q'\\
                                                                 a\not\in Q'
                                                                 \end{array}} X({t(a\alpha)}),$$
and if  $v \not\in Q'_0$, then
$$e^{Q'}_{\rho}(X)(v)=\displaystyle\prod_{\begin{array}{c}  s(a\alpha)=v \\
                                                t(a\alpha)\in Q'\\
                                                 a\not\in Q'
                                                 \end{array}} X({t(a\alpha)}),$$
where the symbols $\alpha$ and $a$, stand respectively, for a path and an arrow with $t(\alpha)=s(a)$. The arrows in $e^{Q'}_{\rho}(X)$ are represented
by the natural projections and if $g: X \to Y$ is a morphism in $\text{Rep}(Q', R)$, then $e^{Q'}_{\rho}(g)$ is just the
obvious map. In particular, if $Q'=\{v\}\subseteq Q$, then for any $R$-module $M$ and any vertex $w\in Q_0$,
$$ e^{v}_{\rho}(M)(w)= \prod_{Q(w,v)} M.$$

\subsection*{The stalk functor} Let $Q$ be a quiver, and $i\in Q_0$ a vertex. The $i$th \emph{stalk functor} $$s_i: R\text{-Mod}\to \text{Rep}(Q, R)$$
is defined as
$$s_i(M)(j):=\left\{
            \begin{array}{ll}
              M, & j=i; \\
             \ 0, & j\neq i.
            \end{array}
          \right.
$$
for each vertex $j\in Q_0$, and
$s_i(M)(a)=0$ for any arrow $a\in Q_1$.

For a given vertex $i$ in a quiver $Q$, we denote
by $Q_1^{i\to *}$ (respectively  $Q_1^{*\to i}$) the set of arrows in $Q$ whose source (respectively  target) is the vertex $i$, that is,
$$Q^{i\rightarrow\ast}_1:=\{a\in Q_1| s(a)=i\} \hspace{0.6cm} \text{and} \hspace{0.6cm} Q^{\ast\rightarrow i}_1:=\{a\in Q_1| t(a)=i\} \eqno{(2.1)}$$
From now on, the letter $Q$ will denote a quiver with the sets of vertices and arrows $Q_0, Q_1$, respectively.
Let $X\in\rep(Q, R)$ be a representation. By the universal property of coproducts, there is a unique homomorphism
$$\varphi_i^X: \bigoplus_{a\in Q^{\ast\rightarrow i}_1}X(s(a))\longrightarrow X(i).  \eqno{(2.2)}$$

Dually, by the universal property of products,  there is a unique homomorphism
$$\psi_i^X: X(i)\longrightarrow\prod_{a\in Q^{i\rightarrow\ast}_1}X(t(a)).  \eqno{(2.3)}$$

Given a subcategory $\mathcal{C}$ of $R$-Mod. We let
$$  \text{Rep}(Q, \mathcal{C})=\{X\in \text{Rep}(Q,R)\hspace{2mm}|\hspace{1mm} X(v)\in\mathcal{C} \text{ for each }v\in Q_0\};  \eqno{(2.4)} $$
  $$ \Psi(\mathcal{C})=\left \{X\in \text{Rep}(Q,R)\hspace{2mm}\begin{array}{ |l} \psi^X_v\text{ is an epimorphism, } X(v) \text{  and}\\
 \Ker(\psi^X_v)\in\mathcal{C} \text{ for each }v\in Q_0\end{array}\right\}.   \eqno{(2.5)} $$

\subsection*{Opposite quivers}
Let  $Q$ be a quiver. By $Q^{\text{op}}=(Q_0, Q_1^{\text{op}})$ we mean a quiver  with the
same set of vertices and the set of reversed arrows, i.e., if $a: i\to j$ is an arrow in $Q_1$, then
 $a^{\text{op}}: j\to i$ is the corresponding  arrow in $Q_1^{\text{op}}$.

\subsection*{Rooted quivers} Let $Q$ be a quiver. As in \cite[2.9]{HP19}, we consider
the transfinite sequence $\{W_{\alpha}\}_{\alpha\text{ }\mathrm{ ordinal}}$ of subsets of $Q_0$ defined as follows:
\begin{enumerate}
\item For the first ordinal $\alpha=0$, set $W_0=\emptyset$;
\item For a successor ordinal $\alpha+1$, set
$$W_{\alpha+1}=\left \{i\in Q_0 \hspace{0.2cm}
\begin{array}{ |l} i \text{ is not the source of any arrow} \\
 a \text{ in }  Q \text{\ with \ } t(a)\not\in\bigcup_{\gamma\leq\alpha}W_\gamma\end{array}\right\};
$$
\item For a limit ordinal $\alpha$, set $W_{\alpha}=\bigcup_{\beta<\alpha}W_{\beta}$.
\end{enumerate}
A quiver $Q$ is called \emph{right rooted} if there is some ordinal $\lambda$ such that $W_{\lambda}=Q_0$; equivalently,
if there is no infinite sequence of arrows of the form $\bullet\to\bullet\to\bullet\to\cdots$  in $Q$ (also see \cite{EEGR09}).
Dually, it follows from \cite[3.6]{EOT04} that a quiver $Q$ is \emph{left rooted} if and only if
it has no infinite sequence of arrows of the form $\cdots \to \bullet\to \bullet \to \bullet$.
Also see Holm and J{\o}rgensen \cite[2.5 and 2.6]{HP19} for the left and right rooted quivers.
We note that  a quiver $Q$ is right
rooted if and only if $Q^{\text{op}}$
is left rooted.


The following result, which is  dual to Holm and Jorgenson \cite[Lemma 2.7 and Corollary 2.8]{HP19}, provides an elegant characterization
of the properties about the transfinite sequence $\{W_{\alpha}\}_{\alpha\text{\ }\mathrm{ ordinal}}$ defined in a right rooted quiver.

\begin{prop}\label{rooted quiver property}
Let $Q$ be a quiver, and  $\{W_{\alpha}\}_{\alpha\text{ }\mathrm{ ordinal}}$ the transfinite sequence, of subsets of $Q_0$, defined as the above. Then the following statements hold:
\begin{enumerate}
  \item[(1)] The transfinite sequence $\{W_{\alpha}\}_{\alpha\text{ }\mathrm{ ordinal}}$ is ascending; that is,
  one has $W_\alpha\subseteq W_\beta$ for each pair of ordinals $\alpha, \beta$ with $\alpha<\beta$. In particular, we have $W_\beta=\bigcup_{\alpha\leq \beta}W_\alpha$  for each ordinal $\beta$.
  \item[(2)] Let $v, w\in Q_0$. If $w\not\in W_{\alpha}$  and $v\in W_{\alpha+1}$ (in particular, if $v\in W_{\alpha}$), then there is
  no arrow $a: v\to w$ belonging to $Q$.
\end{enumerate}
\end{prop}
\begin{proof}
The proof of (1) is dual to that of \cite[Lemma 2.7]{HP19}, and (2) is clear.
\end{proof}

\subsection*{Tensor products of representations}
Inspired by the work of Salarian and Vahed \cite{SV16}, Di, Estrada, Liang and Odabasi \cite[2.1]{DELO23} gave the following definition
of tensor product functors in the category of representations.
For a given representation $X$ in $\text{Rep}(Q, R)$, and a $\mathbb{Z}$-module $G$. The representation
$\Hom(X, G)\in\text{Rep}(Q^{\text{op}}, R^{\text{op}})$ is constructed as follows:

\begin{enumerate}
     \item[$\bullet$] For each vertex $i\in Q^{\text{op}}_0$, set $\Hom(X, G)(i)=\Hom_{\mathbb{Z}}(X(i), G)$;
     \item[$\bullet$] For each arrow $a: i\to j$ in $Q_1^{\text{op}}$,  define $\Hom(X, G)(a)=\Hom_{\mathbb{Z}}(X(a^{\text{op}}), G):
     \Hom_{\mathbb{Z}}(X, G)(i)\to \Hom_{\mathbb{Z}}(X, G)(j),$  where $a^{\text{op}}: j\to i$ is an arrow in $Q$.
\end{enumerate}

It is clear that $\Hom(X, -)$ is a functor from the category $\mathbb{Z}$-Mod of all $\mathbb{Z}$-modules
 to $\text{Rep}(Q^{\text{op}}, R^{\text{op}})$. This functor is left exact and preserves arbitrary products,
so it has a left adjoint from $\text{Rep}(Q^{\text{op}}, R^{\text{op}})$ to $\mathbb{Z}$-Mod, which is denoted $-\otimes_QX$
and will play the role of the tensor product.

\begin{lem} (\cite[Theorem 2.2]{DELO23})\label{adjoint} Let $X$ be a representation in $\text{Rep}(Q, R)$ and $Y$ a representation in $\text{Rep}(Q^{\text{op}}, R^{\text{op}})$. Then
there is a natural isomorphism
$$\Hom_{\mathbb{Z}}(Y\otimes_Q X, G)\cong \Hom_{Q^{\text{op}}}(Y, \Hom(X, G))$$
for each $\mathbb{Z}$-module $G$.
\end{lem}

By Lemma \ref{adjoint}, it is straightforward to verify that the functors $-\otimes_QX$ and $Y\otimes_Q-$ are right exact and preserve arbitrary direct sums and
direct summands. In the following we always let $X^+$ denote the representation $\Hom(X, \mathbb{Q}/\mathbb{Z})\in\text{Rep}(Q^{\text{op}}, R^{\text{op}})$.

\subsection*{Purity in the category of representations}
Recall from \cite[Difinition 2.6 and Lemma 2.19]{GT12} that a short exact sequence of  $R$-modules is called pure if it remains exact after applying   the
tensor product functor $M\otimes_R-$ with any $R^{op}$-module $M$, or equivalently, after applying the Hom functor $\text{Hom}_R(N,- )$ for any finitely presented $R$-module $N$. We are inspired to introduce the following definition.

\begin{df}An exact sequence
 $\eta: 0\to X\to Y\to Z\to 0$ of representations in $\text{Rep}(Q, R)$ is called pure if $S\otimes_Q\eta: 0\to S\otimes_QX\to S\otimes_QY\to S\otimes_QZ\to 0$ is exact for any representation $S\in \text{Rep}(Q^{\text{op}}, R^{\text{op}})$.
 \end{df}

The next result provides a nice characterization of pure exact sequences in $\text{Rep}(Q, R)$ and plays an important role  in our work.

\begin{lem} \label{pure} Let $\eta: 0\to X\to Y\to Z\to 0$ be an exact sequence in $\text{Rep}(Q, R)$. Then  $\eta$ is  pure  if and only if the sequence
$\eta^+: 0\to Z^+\to Y^+\to X^+\to 0$ is splitting exact in $\text{Rep}(Q^{\text{op}}, R^{\text{op}})$.
\end{lem}
\begin{proof} If the sequence   $\eta^+: 0\to Z^+\to Y^+\to X^+\to 0$ in $\text{Rep}(Q^{\text{op}}, R^{\text{op}})$ is splitting exact, then $\Hom_{Q^{\text{op}}}(S, \eta^+)$ is exact  for any representation $S\in \text{Rep}(Q^{\text{op}}, R^{\text{op}})$, now it follows from the following commutative diagram
 $$\xymatrix{
  0 \ar[r]^{} & \Hom_{Q^{\text{op}}}(S, Z^+) \ar[d]_{\cong} \ar[r]^{} & \Hom_{Q^{\text{op}}}(S, Y^+) \ar[d]_{\cong} \ar[r]^{} &\Hom_{Q^{\text{op}}}(S, X^+) \ar[d]_{\cong} \ar[r]^{} &0 \\
   0 \ar[r]^{} & (S\otimes_QZ)^+ \ar[r]^{} &  (S\otimes_QY)^+ \ar[r]^{} &  (S\otimes_QX)^+\ar[r]^{} &0 }$$
 that the sequence $S\otimes_Q\eta$ is exact as the three vertical maps are natural isomorphic, see Lemma \ref{adjoint}. Thus the exact sequence $\eta$ is  pure.

  Conversely, suppose that the exact sequence $\eta$ is  pure, that is, the sequence
$S\otimes_Q\eta$ is  exact for any representation $S\in \text{Rep}(Q^{\text{op}}, R^{\text{op}})$. Then the second row in the diagram above is clearly exact, and again we infer from  Lemma \ref{adjoint}  that the sequence
$$\xymatrix{  0 \ar[r]^{} & \Hom_{Q^{\text{op}}}(S, Z^+)  \ar[r]^{} & \Hom_{Q^{\text{op}}}(S, Y^+) \ar[r]^{} &\Hom_{Q^{\text{op}}}(S, X^+)  \ar[r]^{} &0 }$$ is exact. It is easily seen that $\eta^+$ is exact, and so by taking $S= X^+$, one then gets readily that $\eta^+$ is splitting exact, as desired.
\end{proof}

Recall from \cite[Definition 2.2]{Hos13} that a monomorphism $f: X\to Y$ in $\text{Rep}(Q, R)$ is said to be a pure  monomorphism if
 $f^+: Y^+\to X^+$ is a splitting epimorphism,  and an epimorphism $g: Y\to Z$ in $\text{Rep}(Q, R)$ is said to be a pure epimorphism if
 $g^+: Z^+\to Y^+$ is a splitting monomorphism. In the sense, an exact sequence $ 0\to X\xrightarrow{f} Y\xrightarrow{g} Z\to 0$  in $\text{Rep}(Q, R)$ is  pure  if and only if $f^+: Y^+\to X^+$ is a splitting epimorphism, if and only if  $g^+: Z^+\to Y^+$ is a splitting monomorphism.
Thus, it follows by Lemma \ref{pure} that our definition of the pure exact sequence in $\text{Rep}(Q, R)$ agrees to \cite[Definition 2.2]{Hos13}.

Recall from \cite{AN14} that a representation $X$ is called \emph{fp-injective}, or \emph{absolutely pure} if every
exact sequence $0\to X\to Y$ is pure. The following results are obtained by Enochs, Estrada, Garc\'{\i}a Rozas, Oyonarte and Torrecillas in \cite{EE05, EEGR09, EOT04}, and  Aghasi and Nemati in \cite{AN14}, respectively.

\begin{lem}\label{proflatfp-injective} Let $Q$ be a quiver, and $X$ a representation in $\text{Rep}(Q, R)$.
\begin{enumerate}
  \item If $X$ is projective (respectively, flat), then  $\varphi_i^X$ is a splitting (respectively, pure) monomorphism, and
  $X(i)$ is a projective (respectively, flat) $R$-module for each vertex  $i\in Q_0$. The converse holds provided that the quiver $Q$ is left rooted.
  \item If $X$ is injective, then $\psi_i^X $ is a splitting epimorphism, and  $X(i)$ is an injective $R$-module for each vertex  $i\in Q_0$. The converse holds provided that the quiver $Q$ is right rooted.
 \item If $X$ is  fp-injective, then $\psi_i^X $ is a  pure epimorphism, and
 $X(i)$ is an fp-injective $R$-module for each vertex  $i\in Q_0$. The converse holds provided that the quiver $Q$ is right rooted, and $R$ is left coherent.
\end{enumerate}
\end{lem}

It is evident
that any projective representation is always flat since a flat representation is defined to be a direct limit
of (finitely generated) projective representations.

\vspace{2mm}
{
 Let $\eta: 0\to X\to Y\to Z\to 0$ be an exact sequence in $\text{Rep}(Q, R)$.  Then it is clear that the pure exactness of $\eta$ implies the pure exactness of the sequence $\eta(v): 0\to X(v)\to Y(v)\to Z(v)\to 0$ of $R$-modules for each vertex $v\in Q_0$, but the converse is not true in general. In fact, let $Q$ be a non-discrete quiver, that is, it has at least one arrow, and $X$ a representation.
Then we have an exact sequence
 $$\xi: \xymatrix{
   0 \ar[r]^{} & X \ar[r]^{\hspace{-8mm}f} &  \prod_{v\in Q_0} e^{v}_{\rho}(X(v)) \ar[r]^{\hspace{-5mm}g} &  \prod_{a\in Q_1} e^{s(a)}_{\rho}(X(t(a))) \ar[r]^{} & 0   }$$
 in $\text{Rep}(Q, R)$ with $\xi(v)$ splitting for each vertex $v\in Q_0$, and so $\xi(v)$ is pure exact.
Moreover,  if we let $Q=(\xymatrix{ \bullet \ar[r]_{1\hspace{1.1cm}2} & \bullet })$ and $X=s_2(I)$ be a representation of $Q$
with $I\not=0$ an injective $R$-module, then it is easy to see that the exact sequence
$$\xi^+: \xymatrix{ 0 \ar[r]^{} & [\prod_{a\in Q_1} e^{s(a)}_{\rho}(X(t(a)))]^+ \ar[r]^{\hspace{5mm}g^+} &  [\prod_{v\in Q_0} e^{v}_{\rho}(X(v))]^+ \ar[r]^{\hspace{15mm}f^+} &  X^+ \ar[r]^{} & 0   }$$
restated explicitly as follows
$$\xymatrix{
  0 \ar[r]^{} & 0 \ar[d]_{} \ar[r]^{} & X^+ \ar@{=}[d] \ar@{=}[r] & X^+ \ar[d]_{} \ar[r]^{} & 0 \\
  0 \ar[r]^{} & X^+ \ar@{=}[r] & X^+ \ar[r]^{} & 0 \ar[r]^{} & 0   }$$
isn't splitting, that is, the exact sequence
$0\to s_2(I) \to   e^{2}_{\rho}(I) \to   e^{1}_{\rho}(I) \to 0$ of representations
 in $\text{Rep}(Q, R)$ isn't pure.
}

\section{Strongly fp-injective representations}
In  this section the notion of strongly fp-injective representations is introduced and the properties of such representations are
investigated.

Recall from \cite{Hos13} that a  representation $F$ in $\text{Rep}(Q, R)$ is flat if and only if any epimorphism
with target $F$ is pure, and dually, $J$ is called fp-injective if and only if  any monomorphism
with source $J$ is pure (\cite{AN14}). We see from \cite[Proposition 2.6]{Hos13} that the class of all flat representations  is closed under
taking kernels of epimorphisms. However, it is easily seen that the class of all fp-injective representations isn't closed under
taking cokernels of monomorphisms generally. In fact, the purity of an exact sequence $0\to J\to E\to E/J\to 0 $  of representations with $E$ injective doesn't necessarily imply the  fp-injectivity of $E/J$.

By further pursuit one would see that a representation $X$ in $\text{Rep}(Q, R)$ is flat if and only if there exists a pure acyclic projective resolution  in the category $\text{Rep}(Q, R)$, that is,
there exists a pure exact sequence of representations $\cdots\to P^{-1}\to P^0\to X\to 0$ with each $P^i$ projective. We are inspired to
introduce the notion of  strongly fp-injective representations as the following, the class of these representations is shown to be
closed under taking cokernels of monomorphisms, see Proposition \ref{closeprop1}.

\begin{df}  A representation $X$ is called strongly fp-injective if it admits a pure acyclic injective resolution  $ 0 \rightarrow X \rightarrow E^0 \rightarrow E^1\rightarrow \cdots$ in $\text{Rep}(Q, R)$  with $E^i$ injective for all $i\geq 0$.
\end{df}

\begin{rem} The following statements hold:
\begin{enumerate}
  \item[(1)] It is clear that any injective representation is strongly fp-injective. By  \cite[Theorem 4.2]{AN14} any  strongly fp-injective
  representation is  fp-injective. On the other hand, we note that such implications are strict for general rings.
  \item[(2)] If   $ 0 \rightarrow X \rightarrow E^0 \rightarrow E^1\rightarrow \cdots$  is a pure acyclic injective resolution of $X$, then $\Ker(E^{i} \rightarrow E^{i+1})$ is strongly fp-injective in $\text{Rep}(Q, R)$  for each $i\geq 0$.
\end{enumerate}
\end{rem}

In the following, we will consider some closure properties of the class of strongly fp-injective representations.

\begin{prop}\label{closeprop}
The class of all strongly fp-injective representations is closed under
extensions, finite direct sums and direct summands.
\end{prop}
\begin{proof}
By using the Horse-Shoe Lemma and the Snake Lemma, it follows readily that
 the class of all strongly fp-injective representations is closed under
extensions, and so it is closed under finite direct sums.

To see that the class is closed under direct summands, let  $X\oplus Y$ be  a strongly fp-injective representation, and we will show that
the representation $X$ (or $Y$) is strongly fp-injective.
By  definition there is a pure exact sequence
$$\xymatrix{0 \ar[r]^{} & X\oplus Y \ar[r]^{\alpha} &E^0 \ar[r]^{} & X' \ar[r]^{} &0   }$$
of representations such that $E^0$ is injective and $X'$ is strongly fp-injective. Consider
the push-out diagram with exact rows and columns:
$$ \xymatrix@C=8mm@R=8mm{
  & 0\ar[d]_{}                        &0\ar[d]_{}& &\\
  & X\ar[d]_{\rho_X} \ar@{=}[r]^{} &X\ar[d]_{\beta_X}& &\\
 0 \ar[r]^{} & X\oplus Y \ar[d]_{}\ar[r]^{\alpha}  &E^0\ar[d]_{} \ar[r]^{} &X'\ar@{=}[d]_{}\ar[r]^{} &0 \\
 0 \ar[r]^{} & Y \ar[d]_{}\ar[r]^{}  & X^1\ar[d]_{} \ar[r]^{} &X'\ar[r]^{} &0 \\
            & 0  & 0  & &  }$$
As the morphisms $\rho_X$ and $\alpha$ are pure monomorphisms,   so is $\beta_X$. Thus the exact sequence of representations
$$\xymatrix{0 \ar[r]^{} & X \ar[r]^{\beta_X} &E^0 \ar[r]^{} & X^1 \ar[r]^{} &0   }\eqno{(1)}$$
is pure exact.

By interchanging the roles of $X'$ and $X$, one obtains
a pure exact sequence
$$\xymatrix{0 \ar[r]^{} & Y \ar[r]^{\beta_Y} &E^0 \ar[r]^{} & Y^1 \ar[r]^{} &0   }\eqno{(2)}$$
of representations.
The direct sum of the sequences
 (1) and (2) makes up the upper row, which is easily seen pure exact, in the following commutative diagram:
$$\xymatrix{
  0 \ar[r]^{} & X\oplus Y \ar@{=}[d]_{} \ar[r]^{} & E^0\oplus E^0 \ar[d]_{\theta} \ar[r]^{} & X^1\oplus Y^1 \ar[d]_{} \ar[r]^{} & 0 \\
  0\ar[r]^{} & X\oplus Y \ar[r]^{\alpha} & E^0 \ar[r]^{} & X' \ar[r]^{} & 0  }$$
where $\theta:  E^0\oplus E^0\to  E^0$ is the morphism given by $\theta(i)(a, b)=a+b$ for each vertex $i\in Q_0$.
 By the Snake Lemma we obtain the following commutative diagram with exact rows and columns:
$$\xymatrix@C=8mm@R=8mm{
 & & 0\ar[d]_{}                        &0\ar[d]_{}& \\
 & & E^0\ar[d]_{} \ar@{=}[r]^{} &E^0\ar[d]_{}& \\
 0 \ar[r]^{} &X\oplus Y \ar@{=}[d]_{} \ar[r]^{}  & E^0\oplus E^0\ar[d]_{\theta} \ar[r]^{} &X^1\oplus Y^1\ar[d]_{}\ar[r]^{} &0 \\
 0 \ar[r]^{} & X\oplus Y  \ar[r]^{\alpha}  & E^0\ar[d]_{} \ar[r]^{} &X'\ar[d]_{}\ar[r]^{} &0 \\
             &               & 0  & 0  &  }$$
 It follows that the right-hand column splits as $E^0$ is injective. Since the class of strongly fp-injective representations is closed under
 finite direct sums, one gets that $X^1\oplus Y^1\cong E^0\oplus X'$ is strongly fp-injective.

By replacing $X\oplus Y$ with $X^1\oplus Y^1$ and using the same procedure, one gets a pure acyclic injective resolution
$0\to X \to E^0 \to E^1\to E^2 \to \cdots $
of the representation $X$. This shows that $X$ is a strongly fp-injective.
\end{proof}

\begin{prop}\label{closeprop1}
The class of all strongly fp-injective representations is closed under
taking cokernels of monomorphisms.
\end{prop}
\begin{proof}
Let $0\to X\to Y\to Z\to 0$ be an exact sequence of representations with $X$ and $Y$ strongly fp-injective.
Then there is a pure exact sequence of representations $0\to X\to I\to X'\to 0$ such that $X'$ is strongly fp-injective and $I$ is injective.
Consider the push-out diagram of representations with exact rows and columns:
$$\xymatrix@C=8mm@R=8mm{
 & 0\ar[d]_{}    &0\ar[d]_{}& &\\
 0 \ar[r]^{} &X \ar[d]_{} \ar[r]^{} &Y\ar[d]_{}\ar[r]^{}&Z  \ar@{=}[d]_{} \ar[r]^{}& 0\\
 0 \ar[r]^{} &I \ar[d]_{} \ar[r]^{}  & U\ar[d]_{} \ar[r]^{} &Z \ar[r]^{} &0 \\
             &X'\ar[d]    \ar@{=}[r]^{}  & X'\ar[d]_{}&& \\
             & 0             & 0  &  &  }$$
As  the class of strongly fp-injective representations is closed under extensions (see Proposition \ref{closeprop}), the exactness of
the middle column implies that $U$ is strongly fp-injective. Since $I$ is injective,  it follows that the middle row splits. Now we get that
$I\oplus Z\cong U$ is strongly fp-injective, and then so is $Z$ by  Proposition \ref{closeprop}.
\end{proof}

\begin{prop}\label{closed p1} Let  $0\rightarrow X\rightarrow Y\rightarrow Z\rightarrow 0$ be a pure exact sequence of  representations.
If $Y$ and $Z$ are strongly fp-injective, then so is $X$.
\end{prop}
\begin{proof}
By the definition, there is  a pure exact sequence of representations $0\rightarrow Y\rightarrow E\rightarrow V\rightarrow 0$ with $E$ injective and  $V$ strongly fp-injective. Take into account the following push-out diagram of representations:
$$\xymatrix{
               &                      & 0\ar[d]_{}            & 0 \ar[d]_{}              &  \\
  0 \ar[r]^{} & X \ar@{=}[d]_{} \ar[r]^{} & Y \ar[d]_{} \ar[r]^{} & Z \ar[d]_{} \ar[r]^{} & 0\\
  0 \ar[r]^{} & X           \ar[r]^{} & E \ar[d]_{} \ar[r]^{} & U   \ar[d]_{} \ar[r]^{} & 0 \\
              &                       & V \ar[d]_{} \ar@{=}[r]^{} & V   \ar[d]_{}             &  \\
              &                       & 0                     & 0                         & }
$$
The exactness of the right-hand column implies that $U$ is strongly fp-injective as
the class of strongly fp-injective representations is closed under extensions, see Proposition \ref{closeprop}. It is then easy to see that the middle row
is pure exact, and so it is readily checked that $X$ is strongly fp-injective.
\end{proof}

As an immediate consequence of Proposition \ref{closeprop1} and Proposition \ref{closed p1} respectively, one has the following corollaries.

\begin{cor} \label{closecokernel} A representation $X$ is strongly fp-injective if and only if any exact sequence  $ 0 \rightarrow X \rightarrow E^0 \rightarrow E^1\rightarrow \cdots$ is pure, where $E^i$ is injective for all $i\geq 0$.
\end{cor}

\begin{cor} Any fp-injective representation of  strongly fp-injective dimension at most 1 is
strongly fp-injective.
\end{cor}

\begin{rem} (1) Given a pure short exact sequence $0\to X\to Y\to Z\to 0$ of representations, it is known from Propositions \ref{closeprop}-\ref{closed p1} that,
if any two representations $X, Y$ and $Z$ are strongly fp-injective, then so is the third.

(2) Note that the class of all strongly fp-injective representations is neither
 closed under taking  pure subrepresentations nor closed under taking  pure quotients of representations, see Example \ref{exa1} below.
\end{rem}

Recall from \cite{LGO} that an $R$-module $M$ is called \textbf{strongly fp-injective} if,
for any finitely presented $R$-module $N$, $\Ext_R^i(N,M)=0$ for all $i\geq 1$.
Emmanouil and  Kaperonis  characterized certain complexes of strongly fp-injective modules, see \cite{EK}. In particular,
if $R$ is left coherent, then the strongly fp-injective modules are exactly the fp-injective modules; and if
$R$ is left noetherian, then they are exactly the injective modules.
 We will show that such modules are closely connected to  strongly fp-injective representations.

\begin{exa} \label{exa1} {\rm Let $Q$  be the quiver
$\xymatrix{ \bullet \ar[r]_{1\hspace{1.1cm}2} & \bullet }$, and let $M$ be
an fp-injective $R$-module which is not  strongly fp-injective (such $M$ exists by \cite[Example 1]{LGO}).
Consider the exact sequence of $R$-modules $ 0 \rightarrow M \rightarrow E^0 \rightarrow E^1\rightarrow \cdots$
 with $E^i$ injective for all $i\geq 0$. It is evident that the sequence is not pure exact in $R$-Mod,
 and  so we deduce that the induced sequence of representations
$$\xymatrix{0 \ar[r]^{} & e^1_{\rho}(M) \ar[r]^{} & e^1_{\rho}(E^0) \ar[r]^{} & e^1_{\rho}(E^1) \ar[r]^{} & e^1_{\rho}(E^2)\ar[r]^{}&\cdots  }$$
 is not pure exact as well. Hence $e^1_{\rho}(M)=(M \to 0)$ is not a strongly fp-injective representation though
 it is a pure subrepresentation of $e^1_{\rho}(E^0)$, where $e^1_{\rho}(E^i)=(E^i \to 0)$ are injective (and so  strongly fp-injective) in $\text{Rep}(Q, R)$ for all $i\geq 0$, see Lemma \ref{proflatfp-injective}.}
\end{exa}


{We call a quiver $Q$   \emph{locally target-finite} if there are only
finitely many arrows in $Q$ from any given vertex to some others, that is, if the set $Q^{v\rightarrow\ast}_1$
is finite for any vertex $v\in Q_0$. Next, we intend to clarify the relations between  strongly fp-injective representations and  strongly fp-injective modules.}

\begin{thm}\label{strongly FP} Let  $X$ be a strongly fp-injective representation in $\text{Rep}(Q, R)$. {If the
quiver $Q$ is locally target-finite}, then the following statements hold.
\begin{enumerate}
  \item For each vertex  $v\in Q_0$,  $X(v)$ is a strongly fp-injective $R$-module.
  \item  For each vertex  $v\in Q_0$, the homomorphism $\psi_v^X: X(v)\longrightarrow\prod_{a\in Q^{v\rightarrow\ast}_1}X(t(a)) $ induced by
  $X(a): X(v)\longrightarrow X(t(a))$ is a pure epimorphism.
  \end{enumerate}
\end{thm}
\begin{proof}
(1) By definition, there is a pure exact sequence of representations
$$\xymatrix@C=8mm@R=10mm{0 \ar[r]^{} & X \ar[r]^{f} & E_0 \ar[r]^{f_0} &E_1 \ar[r]^{f_1} & E_2 \ar[r]^{} &\cdots  },$$
where all $E_i$ are injective representations.
It follows from Lemma \ref{proflatfp-injective}  that all $E_i(v)$ are injective $R$-modules. Note that the induced sequence of $R$-modules
$$\xymatrix@C=8mm@R=10mm{0 \ar[r]^{} & X(v) \ar[r]^{f(v)} & E_0(v) \ar[r]^{f_0(v)} &E_1(v) \ar[r]^{f_1(v)} & E_2(v) \ar[r]^{} &\cdots  }$$
is pure exact as well,  and so $\Ext^{n}_R(N, X(v))=0$ for any finitely presented $R$-module $N$ and all  $n\geq 1$.
Consequently, $X(v)$ is a strongly fp-injective $R$-module. This result holds true for any quiver $Q$ (not necessarily locally target-finite).

\vskip0.1in
(2) By the assumption, there is a pure monomorphism $f: X\to I$ of representations with $I$ injective.  Consider the following commutative diagram for any vertex $v\in Q_0$:
$$\xymatrix@C=12mm{  X(v)\ar[d]_{f(v)} \ar[r]^{\hspace{-4mm}\psi^{X}_v} & \displaystyle\prod_{a\in Q^{v\rightarrow\ast}_1}\hspace{-3mm}X(t(a))
 \ar[d]_{\prod f(t(a))} \ar[r]^{\hspace{3mm}\pi^{X}_{t(a)}} &  \ar[d]_{f(t(a))}  X(t(a))  \\
I(v)  \ar[r]^{\hspace{-4mm}\psi^{I}_v}  &\displaystyle\prod_{a\in Q^{v\rightarrow\ast}_1}\hspace{-3mm}I(t(a)) \ar[r]^{\hspace{3mm}\pi^{I}_{t(a)}}   & I(t(a)) }$$
where  $\pi^{X}_{t(a)}: \prod_{a\in Q^{v\rightarrow\ast}_1}X(t(a))\to X(t(a))$ denotes the $t(a)$-th coordinate homomorphism, $\pi^{I}_{t(a)}$  is defined similarly. By the properties of product, it is obvious that $\pi^{X}_{t(a)}\psi^{X}_v=X(a)$ and $\pi^{I}_{t(a)}\psi^{I}_v=I(a)$ for each arrow $a\in Q^{v\rightarrow\ast}_1$.

Since the homomorphism $\psi^{I}_v$ is splitting epimorphic, it follows that there is a homomorphism
$\phi_v: \prod_{a\in Q_1^{v\to *}}I(t(a))\to I(v)$ such that
$\psi^{I}_v\phi_v=1_{\prod_{a\in Q_1^{v\to *}}I(t(a))}$.
Since $X$ is strongly fp-injective, it follows that the sequence  $I^+ \xrightarrow{f^+} X^+ \rightarrow 0$
splits, that is,  there exists a homomorphism $g: X^+\to I^+$ such that $f^+g=1_{X^+}$, and also $g(v)X^+(a)=I^+(a)g(t(a))$ for any arrow $a\in Q_1^{v\to *}$. Now consider the following commutative diagram:
\begin{center}
\begin{tikzpicture}
\matrix(a)[matrix of math nodes,
          row sep=3.0em,
          column sep=4.5em,
          text height=1.5ex, text depth=0.5ex]
{I^+(t(a)) & {[}\displaystyle\prod_{a\in Q^{v\rightarrow\ast}_1}\hspace{-3mm}I(t(a)){]}^+ & I^+(v) \\
X^+(t(a)) & {[}\displaystyle\prod_{a\in Q^{v\rightarrow\ast}_1}\hspace{-4mm}X(t(a)){]}^+ & X^+(v) \\};
\path[->, font=\scriptsize]
(a-1-1) edge node[auto] {$[\pi^{I}_{t(a)}]^+$} (a-1-2)
(a-1-2) edge node[auto] {$(\psi^I_v)^+$} (a-1-3)
(a-1-1) edge node[auto] {$f^+(t(a))$} (a-2-1)
(a-1-2) edge node[auto] {$[\prod f(t(a))]^+$} (a-2-2)
(a-1-3) edge node[auto] {$f^+(v)$} (a-2-3)
(a-2-1) edge  [bend right=30]  node[] {$\begin{array}{c} \null \\ {X^+(a)} \end{array}$} (a-2-3)
(a-1-1) edge  [bend left=30]  node[] {$\begin{array}{c} {I^+(a)}\\ \null \end{array}$} (a-1-3)
(a-2-1) edge node[auto] {$[\pi^{X}_{t(a)}]^+$} (a-2-2)
(a-2-2) edge node[auto] {$(\psi^X_v)^+$} (a-2-3);
\end{tikzpicture}
\end{center}
Set $\zeta:= [\prod f(t(a))]^+\circ\phi_v^+\circ g(v):X^+(v)\to [\displaystyle\prod_{a\in Q^{v\rightarrow\ast}_1}\hspace{-3mm}X(t(a))]^+$. As the set $Q^{v\rightarrow\ast}_1$ is finite, for
each element $Y$  in $[\displaystyle\prod_{a\in Q^{v\rightarrow\ast}_1}\hspace{-3mm}X(t(a))]^+\cong\displaystyle\bigoplus_{a\in Q^{v\rightarrow\ast}_1}\hspace{-3mm}X^+(t(a))$, we have $Y=\sum_{a\in Q^{v\rightarrow\ast}_1}[\pi^{X}_{t(a)}]^+(Y_{t(a)})$ with each $Y_{t(a)}\in X^+(t(a)$, and then
\begin{equation*}
\begin{split}\zeta\circ (\psi^X_v)^+(Y)&=\zeta\circ(\psi^X_v)^+\big(\sum_{a\in Q^{v\rightarrow\ast}_1}[\pi^{X}_{t(a)}]^+(Y_{t(a)})\big)\\
&=\sum_{a\in Q^{v\rightarrow\ast}_1}\zeta\circ\big((\psi^X_v)^+\circ[\pi^{X}_{t(a)}]^+\big)(Y_{t(a)})\\
&=\sum_{a\in Q^{v\rightarrow\ast}_1}\zeta\circ X^+(a)(Y_{t(a)})\\
&=\sum_{a\in Q^{v\rightarrow\ast}_1}[\prod f(t(a))]^+\circ\phi_v^+\circ\big( g(v)\circ X^+(a)\big)(Y_{t(a)})\\
&=\sum_{a\in Q^{v\rightarrow\ast}_1}[\prod f(t(a))]^+\circ\phi_v^+\circ \big(I^+(a)\circ g(t(a))\big)(Y_{t(a)})\\
&=\sum_{a\in Q^{v\rightarrow\ast}_1}[\prod f(t(a))]^+\circ\phi_v^+\circ \big((\psi^I_v)^+\circ[\pi^{I}_{t(a)}]^+\circ g(t(a))\big)(Y_{t(a)})\\
&=\sum_{a\in Q^{v\rightarrow\ast}_1}[\prod f(t(a))]^+\circ\big((\phi_v^+\circ (\psi^I_v)^+\big)\circ[\pi^{I}_{t(a)}]^+\circ g(t(a))(Y_{t(a)})\\
&=\sum_{a\in Q^{v\rightarrow\ast}_1}[\prod f(t(a))]^+\circ\big( \psi^I_v\circ\phi_v\big)^+\circ[\pi^{I}_{t(a)}]^+\circ g(t(a))(Y_{t(a)})\\
&=\sum_{a\in Q^{v\rightarrow\ast}_1}\big([\prod f(t(a))]^+\circ[\pi^{I}_{t(a)}]^+\big)\circ g(t(a))(Y_{t(a)})\\
&=\sum_{a\in Q^{v\rightarrow\ast}_1}\big([\pi^{X}_{t(a)}]^+\circ f^+(t(a))\big)\circ g(t(a))(Y_{t(a)})\\
&=\sum_{a\in Q^{v\rightarrow\ast}_1}[\pi^{X}_{t(a)}]^+\circ \big(f^+(t(a))\circ g(t(a))\big)(Y_{t(a)})\\
&=\sum_{a\in Q^{v\rightarrow\ast}_1}[\pi^{X}_{t(a)}]^+(Y_{t(a)})\\
&=1_{[\prod_{a\in Q^{v\rightarrow\ast}_1}X(t(a))]^+}(Y).\\
\end{split}
\end{equation*}
This shows $(\psi^{X}_v)^+$ a  splitting monomorphism, and so $\psi^{X}_v: X(v)\to  \prod_{a\in Q^{v\rightarrow\ast}_1}X(t(a))$
is a pure epimorphism, as desired.
\end{proof}

\begin{rem} We use \cite[Example 3]{EEGR09} given by  Enochs, Estrada and Garcia Rozas, to show that a representation satisfying conditions (1) and (2) in Theorem \ref{strongly FP} is not necessarily strongly fp-injective.
\end{rem}

\begin{exa}{\rm
Consider the non-right rooted quiver $Q$
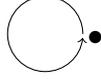
\begin{center}
\begin{tikzpicture}
\draw[->] (1,1) arc (0:354:0.5)node[right] {$\hspace{-0.08cm}\bullet$};
\end{tikzpicture}
\end{center}
with one vertex $v$ and an arrow $\alpha$ from $v$ to itself, and the category of representations of $Q$ by $\mathbf{k}$-vector spaces ($\mathbf{k}$ is a field).
In this case Rep($Q, \mathbf{k}$) is equivalent to the category  $\mathbf{k}[x]$-Mod of $\mathbf{k}[x]$-modules.  If we
take the representation $X$ with $\mathbf{k}[x,x^{-1}]$ in the vertex $v$ and the morphism $X(\alpha):\mathbf{k}[x,x^{-1}]\to \mathbf{k}[x,x^{-1}]$ given
by multiplying $x$. It is clear that $X$ satisfies (1) and (2) in Theorem \ref{strongly FP}, but $\mathbf{k}[x,x^{-1}]$ is not divisible
as a $\mathbf{k}[x]$-module, so it can't be injective. Since the ring $\mathbf{k}[x]$ is noetherian, the notions of strongly fp-injective $\mathbf{k}[x]$-modules and injective $\mathbf{k}[x]$-modules
coincide, and so $X$ is not strongly fp-injective as a $\mathbf{k}[x]$-module.}
\end{exa}

In the next, we will prove that a representation satisfying conditions (1) and (2) in Theorem \ref{strongly FP} is always strongly fp-injective when the quiver $Q$ is right-rooted {and locally target-finite}. To achieve this,  we need the following lemma.

\begin{lem}\label{split diagram} Assume that the commutative diagram of $R$-modules $$\xymatrix{
  0 \ar[r] & X \ar[d]_{f} \ar[r]^{\mu} & Y \ar[d]_{g} \ar[r]^{} & Z \ar[d]_{h} \ar[r]^{} &0 \\
  0 \ar[r] & L \ar[r]^{\nu} & M \ar[r]^{} & N \ar[r]^{} & 0 }$$
has two exact rows split with $r\mu=1$,  where $r: Y\rightarrow X$. If the homomorphism $h$ is monic, and  every homomorphism
$Z\to L$ can be extended to $N\to L$, then there exists a left-inverse $s: M\rightarrow L$ of the homomorphism $\nu$ such that $fr=sg$.
\end{lem}
\begin{proof}  Since the diagram has two exact rows split, we can write it as
$$\xymatrix@C=10mm@R=10mm{
  X \ar[d]_{f} \ar[r]^{\Small\left(
         \begin{matrix}
           1 \\
           0 \\
         \end{matrix}
       \right)
  } & X\oplus Z \ar[d]_{g} \ar[r]^{\big{(0,1)}} & Z \ar[d]_{h} \\
  L \ar[r]^{\Small\left(
         \begin{matrix}
           1 \\
           0 \\
         \end{matrix}
       \right)
  } & L\oplus N \ar[r]^{\big{(0,1)}} & N  }
$$
and  $r={\big{(1,0)}}:X\oplus Z \to X$.   Note that the homomorphism $g: X\oplus Z\rightarrow L\oplus N$  has the form $g=\begin{pmatrix}
     f & k \\
     0 & h \\
   \end{pmatrix}$,
where $k: Z\rightarrow L$. Since $h: Z\to N$ is monic and $k: Z\rightarrow L$ can be extended to $\alpha: N\to L$ by the assumption,  that is, $k=\alpha h$, one has $sg=fr$ with $s=(1,-\alpha)$, as desired.
\end{proof}

The following result provides an explicit characterization of  strongly fp-injective representations of a {locally target-finite and}  right rooted quiver.

\begin{thm}\label{strongly FP right} Let  $Q$ be a {locally target-finite and} right rooted quiver. Then a representation $X$ in $\text{Rep}(Q, R)$ is strongly fp-injective if and only if the following statements hold.
\begin{enumerate}
  \item For each vertex  $i\in Q_0$,  $X(i)$ is a strongly fp-injective $R$-module.
  \item  For each vertex  $i\in Q_0$, the homomorphism $\psi_i^X: X(i)\longrightarrow\prod_{a\in Q^{i\rightarrow\ast}_1}X(t(a)) $  is a pure epimorphism.
    \end{enumerate}
\end{thm}
\begin{proof}
We need only to show the sufficiency  since the necessity follows from Theorem \ref{strongly FP}. Let $X$ be
a representation  in $\text{Rep}(Q, R)$ satisfying the conditions (1) and (2).
In the following, we will show that any exact sequence of representations $0\to X \to E \to C\to 0  $ with $E$ injective is
pure, and $C$ has the same properties as $X$. Then, inductively, we get readily that $X$ admits a pure acyclic injective resolution
$0 \to X \to E^0 \to E^1 \to E^2 \to\cdots$, and so $X$ is a strongly fp-injective representation.

Let  $0\to X \xrightarrow{f} E \xrightarrow{p} C\to 0$ be an exact sequence in  $\text{Rep}(Q, R)$ with $E$ injective.
Then for each vertex $i\in Q_0$, it is easily seen that the sequence of $R$-modules
$$\xymatrix{0 \ar[r]^{} & X(i) \ar[r]^{f(i)} & E(i) \ar[r]^{p(i)} &C(i) \ar[r]^{} &0  }$$
is pure exact with $E(i)$ injective  as $X(i)$ is a strongly fp-injective $R$-module, and so $C(i)$
is strongly fp-injective as the class of strongly fp-injective $R$-modules is closed under
taking cokernels of monomorphisms. Now consider the following commutative diagram with exact rows induced by the short exact sequence $0\to X \xrightarrow{f} E \xrightarrow{p} C\to 0$ of representations at the vertex $i\in Q_0$,
$$\xymatrix@C=6mm{
  0 \ar[r]^{} & X(i) \ar[d]_{\psi_i^X} \ar[r]^{f(i)} & E(i) \ar[d]_{\psi_i^{E}} \ar[r]^{p(i)} & C(i) \ar[d]_{\psi_i^{C}} \ar[r]^{} & 0 \\
  0\ar[r]^{} & \prod_{a\in Q^{i\rightarrow\ast}_1}X(t(a)) \ar[r]^{} & \prod_{a\in Q^{i\rightarrow\ast}_1}E(t(a)) \ar[r]^{} & \prod_{a\in Q^{i\rightarrow\ast}_1}C(t(a)) \ar[r]^{} & 0   }$$
Since the class of strongly fp-injective  $R$-modules is closed under direct products, it is easy to see that the bottom row is pure exact in $R$-Mod. Noting that $\psi_i^{E}$ is a splitting epimorphism, one infers readily that $\psi_i^{C}$ is a pure epimorphism. This proves that $C$ has the same properties as $X$.
To show the purity of the exact sequence of representations
$$\xymatrix{0\ar[r]^{} &X \ar[r]^{f} & E \ar[r]^{p} & C \ar[r]^{} &0  },$$
it suffice  to show that the sequence
$$\xymatrix{0\ar[r]^{} & C^+ \ar[r]^{p^+} & E^+ \ar[r]^{f^+} & X^+ \ar[r]^{} &0  }$$
is splitting in $\text{Rep}(Q^{\text{op}}, R^{\text{op}})$ by Lemma \ref{pure}.
In the next, we will prove that there exist homomorphisms $\gamma(v):  E^+(v) \to  C^+(v)$ for all $v\in Q_0$ such that
$\gamma:  E^+ \to  C^+$ is a morphism of representations satisfying $\gamma\circ p^+=1_{C^+}$.

As $Q$ is right rooted, there  exists some ordinal $\lambda$ such that $Q_0=W_\lambda=\bigcup_{\alpha\leq\lambda}W_\alpha$, see \cite[2.9]{HP19}.
We will show the result by transfinite induction on $\lambda$.
The induction start for $\lambda=0$ is nothing to prove since $W_0$ is empty. For $\lambda=1$, there are no arrows in $Q_1$,
since the sequence $0\to C^+(v) \xrightarrow{p^+(v)}  E^+(v)\xrightarrow{f^+(v)}  X^+(v) \to 0$
is splitting for each $v\in Q_0=W_1$, the result follows readily.

Now we assume that $\lambda>1$ is a limit ordinal and  that the
homomorphism $\gamma(w):  E^+(w) \to  C^+(w)$ has been built such that the following conditions hold for each  $\delta<\lambda$:
\begin{enumerate}
  \item[(I)] $\gamma(w)\circ p^+(w)=1_{ C^+(w)}$ for each $w\in W_\delta$;
  \item[(II)] $\gamma(v)\circ E(a)^+= C(a)^+\circ\gamma(w)$  for each arrow $a: v\to w$ in $Q_0$ with $v, w\in W_\delta$.
\end{enumerate}
We will build a homomorphism $\gamma(v):  E^+(v) \to  C^+(v)$ for any given vertex $v\in W_\lambda$ such that the  conditions (I) and (II) hold for the  ordinal $\lambda$. However, this is clear since we have $W_\lambda=\bigcup_{\delta<\lambda}W_\delta$ by our assumption that $\lambda$ is a limit ordinal.

It remains to consider the case that $\lambda=\delta+1$ is a successor ordinal. We assume again that
the homomorphism $\gamma(w):  E^+(w) \to  C^+(w)$ has been built such that the conditions  (I) and (II) hold for $\delta$.
We must build a homomorphism $\gamma(v):  E^+(v) \to  C^+(v)$  for any given vertex $v\in W_{\delta+1}$ such that  the conditions  (I) and (II) hold for $\delta+1$.

Without loss of generality, we may assume that $v\in W_{\delta+1}\setminus \bigcup_{\alpha\leq \delta}W_\alpha$.
Since we have
$$W_{\delta+1}=\left \{v\in Q_0 \hspace{0.2cm}
\begin{array}{ |l} v \text{ is not the source of any arrow} \\
 a \text{ in }  Q \text{\ with \ } t(a)\not\in\bigcup_{\alpha\leq \delta}W_\alpha\end{array}\right\},
$$
it follows that  $t(a)\in\bigcup_{\alpha\leq \delta}W_\alpha$ for every arrow $a\in Q^{v\rightarrow\ast}_1$. Therefore, we may let $a: v\to w$ be such an arrow, that is, we have $w=t(a)\in W_\delta$, see Proposition \ref{rooted quiver property}. Now consider the commutative diagram  of $R^{op}$-modules with exact rows :
$$\xymatrix@C=6mm@R=6mm{
 0\ar[r] & C^+(w) \ar[d]^{(\pi^{C}_w)^+} \ar[r]^{p^+(w)} & E^+(w) \ar[d]^{(\pi^{E}_w)^+} \ar[r]^{f^+(w)} & X^+(w) \ar[d]^{(\pi^X_w)^+} \ar[r]^{} & 0 \\
  0\ar[r] &[\prod_{a\in Q^{v\rightarrow\ast}_1}C(t(a))]^+ \ar[d]^{(\psi^{C}_v)^+} \ar[r]^{\eta} &[\prod_{a\in Q^{v\rightarrow\ast}_1}E(t(a))]^+ \ar[d]^{(\psi^{E}_v)^+} \ar[r]^{\hspace{-2mm}h} & [\prod_{a\in Q^{v\rightarrow\ast}_1}X(t(a))]^+ \ar[d]^{(\psi^X_v)^+} \ar[r]^{ } & 0  \\
 0\ar[r] & C^+(v)  \ar[r]^{p^+(v)} & E^+(v) \ar[r]^{f^+(v)} & X^+(v) \ar[r]^{} & 0   }
$$
where  $\pi^{C}_{t(a)}: \prod_{a\in Q^{v\rightarrow\ast}_1}C(t(a))\to C(t(a))$ denotes the $t(a)$-th coordinate homomorphism with $t(a)=w$, $\pi^{X}_w$  and $\pi^{E}_w$ are defined similarly. By our assumption, there are homomorphisms $\gamma(w):  E^+(w)\to  C^+(w)$ such that $\gamma(w)\circ p^+(w)=1_{C^+(w)}$ for each $w=t(a)$ with $a\in Q^{v\rightarrow\ast}_1$.
{By the assumption that the quiver $Q$ is locally target-finite, we have the set $Q^{v\rightarrow\ast}_1$  finite. This induces
an isomorphism $[\prod_{a\in Q^{v\rightarrow\ast}_1}M(t(a))]^+\cong \oplus_{a\in Q^{v\rightarrow\ast}_1}M^+(t(a))$ for any representation $M$ of $Q$.
If we set  $$\eta':=\oplus_{a\in Q^{v\rightarrow\ast}_1}\gamma(t(a)): [\prod_{a\in Q^{v\rightarrow\ast}_1}E(t(a))]^+\longrightarrow [\prod_{a\in Q^{v\rightarrow\ast}_1}C(t(a))]^+,$$ then it is easy to check}
$\eta'\circ\eta=1_{[\prod_{a\in Q^{v\rightarrow\ast}_1}C(t(a))]^+}$ and $(\pi^{C}_w)^+\circ\gamma(w)=\eta'\circ(\pi^{E}_w)^+.$
Notice further that the homomorphism $(\psi^X_v)^+$ is a splitting monomorphism
as $$\psi_v^X: X(v)\longrightarrow\prod_{a\in Q^{v\rightarrow\ast}_1}X(t(a)) $$   is a pure epimorphism by the condition, one infers from Lemma \ref{split diagram} that
there is a homomorphism $\gamma(v): E^+(v)\to C^+(v)$ such that
$\gamma(v)\circ p(v)^+=1_{C^+(v)}$ and $(\psi^{C}_v)^+\circ\eta'=\gamma(v)\circ(\psi^{E}_v)^+.$ That is, we have constructed
a homomorphism $\gamma(v):  E^+(v) \to  C^+(v)$  for any given vertex $v\in W_{\delta+1}$ such that  the conditions  (I) and (II) hold for $\delta+1$.
Thus, inductively, we can prove that there exists a morphism $\gamma:  E^+ \to  C^+$ of representations such that  $\gamma\circ p^+=1_{C^+}$.
This implies that the sequence
$$\xymatrix{0 \ar[r]^{} & C^+  \ar[r]^{p^+} &  E^+ \ar[r]^{f^+} & X^+ \ar[r]^{} & 0}$$  is splitting, as desired.
\end{proof}

\begin{cor}\label{strongly ccor} Let  $Q$ be a {locally target-finite and}  right rooted quiver.  Then  the class of strongly fp-injective  representations of  $Q$  is closed under direct products.
\end{cor}
\begin{proof} By  Theorem \ref{strongly FP right}, it is easily seen that if the quiver $Q$ is {locally target-finite and} right rooted, then a representation $X$ of  $Q$ is strongly fp-injective if and only if  the $R$-modules $X(i)$ and $\Ker(\psi_i^X)$ are strongly fp-injective, and
 the homomorphism $\psi_i^X$   is  epimorphic for each vertex  $i\in Q_0$. Since the the class of strongly fp-injective
 $R$-modules  is closed under direct products, one gets the result readily.
\end{proof}

Recently, series work obtained by Holm and Jorgensen \cite{HP19}, and  Odabasi \cite{O19} implies that  there are various ways of lifting
a cotorsion pair in the category $R$-Mod of $R$-modules to a  cotorsion pair in  the category $\text{Rep}(Q, R)$ of representations.
Following the suggestion
in \cite[Page 4]{BHP24}, an $R$-module $K$ is called \emph{weakly fp-projective} if $\Ext^1_R(K, J)=0$ for all
strongly fp-injective $R$-modules $J$. It follows from \cite[Theorem 3.4]{LGO} that
$(\mathcal{K}, \mathcal{J})$ is a hereditary complete cotrosion pair over any ring $R$, where $\mathcal{K}$  and $\mathcal{J}$
denote  the classes of all weakly fp-projective and  strongly fp-injective $R$-modules, respectively.
We note that if the quiver $Q$  is {locally target-finite and} right rooted, then $\Psi(\mathcal{J})$ (see the notation (2.5) in Section 2) is exactly the class of
strongly fp-injective representations of $Q$.

\vskip0.1in
The following result follows directly from \cite[Theorem B]{HP19} and \cite[Theorem B]{O19}.

\begin{lem}\label{cotorsion pair KJ} Let $Q$ be a right rooted quiver. Then  $(\text{Rep}(Q,\mathcal{K}), \Psi(\mathcal{J}))$ forms a
hereditary complete cotorsion pair  in $\text{Rep}(Q, R)$, where $\text{Rep}(Q,\mathcal{K})$ denotes the class of all representations with each component in $\mathcal{K}$.
\end{lem}

It is known that the right adjoint preserves injective objects. In the following  we will see that the right adjoint $e^{Q'}_{\rho}$ of the restriction functor $e^{Q'}$  also preserves  strongly fp-injective representations.

\begin{prop}\label{sub-quiver fp} Let $Q'$ be a subquiver of $Q$, and $X$ be a strongly fp-injective representation of $Q'$. If $Q$ is {locally target-finite and} right rooted, then $e^{Q'}_{\rho}(X)$ is a strongly fp-injective representation of $Q$.
\end{prop}
\begin{proof} Since $X$ is strongly fp-injective in $\text{Rep}(Q', R)$, there is a pure exact sequence
$$\xymatrix{0\ar[r]^{} &X \ar[r]^{} & E^0 \ar[r]^{} & E^1 \ar[r]^{} &\cdots  }$$
in $\text{Rep}(Q', R)$ with each $E^i$ injective.
This induces that the sequence
$$\xymatrix{0\ar[r]^{} &e^{Q'}_{\rho}(X) \ar[r]^{} & e^{Q'}_{\rho}(E^0) \ar[r]^{} & e^{Q'}_{\rho}(E^1) \ar[r]^{} & \cdots   }\eqno{(3.2)}$$
in $\text{Rep}(Q, R)$ is exact with each $e^{Q'}_{\rho}(E^i)$ injective. To see that $e^{Q'}_{\rho}(X)$ is strongly fp-injective in $\text{Rep}(Q, R)$,
it suffice to show that the sequence $(3.2)$ remains exact after applying the functor $\Hom_{Q}(K, -)$
 for any representation $K\in\text{Rep}(Q,\mathcal{K})$, see Lemma \ref{cotorsion pair KJ}, where $\mathcal{K}$
is  the class of all weakly fp-projective  $R$-modules.
Noting  that $Q'$ is right rooted as so is $Q$, we get by Lemma \ref{cotorsion pair KJ} that the cotorsion pair
$(\text{Rep}(Q',\mathcal{K}), \Psi(\mathcal{J}))$
is hereditary complete in $\text{Rep}(Q', R)$,  where $\mathcal{J}$
is  the class of all  strongly fp-injective $R$-modules. Moreover, it is clear that $e^{Q'}(K)$ belongs to
 $\text{Rep}(Q',\mathcal{K})$. Thus, we have the following commutative diagram of representations with the top row exact:
$$\xymatrix@C=5mm{
  0 \ar[r]^{} & \Hom_{Q'}(e^{Q'}(K), X) \ar[d]_{\cong} \ar[r]^{} & \Hom_{Q'}(e^{Q'}(K), E^0) \ar[d]_{\cong} \ar[r]^{} & \Hom_{Q'}(e^{Q'}(K), E^1) \ar[d]_{\cong} \ar[r]^{} & \cdots \\
  0 \ar[r]^{} & \Hom_{Q}(K, e^{Q'}_{\rho}(X)) \ar[r]^{} & \Hom_{Q}(K, e^{Q'}_{\rho}(E^0)) \ar[r]^{} & \Hom_{Q}(K, e^{Q'}_{\rho}(E^1)) \ar[r]^{} & \cdots  }$$
Now  the adjoint isomorphism $$\Hom_{Q'}(e^{Q'}(K), Y)\cong\Hom_{Q}(K, e^{Q'}_{\rho}(Y)),$$ for any representation $Y\in\text{Rep}(Q,R)$, yields that the bottom row is also exact,  and this implies that $\Ext_{Q}^{i\geq 1}(K, e^{Q'}_{\rho}(X))=0$, that is, $e^{Q'}_{\rho}(X)\in\Psi(\mathcal{J})$ is strongly fp-injective in $\text{Rep}(Q, R)$.
\end{proof}

The above result gives the following example.

\begin{exa} Let  $Q$ be a {locally target-finite and} right rooted quiver, and $v\in Q_0$ a vertex. If $M$ is a strongly fp-injective $R$-module, then $e^v_{\rho}(M)$ is a  strongly fp-injective  representation of  $Q$.
\end{exa}

The following result is interesting independently.

\begin{prop} \label{Stability} A representation $X$ in $\text{Rep}(Q, R)$ is
strongly fp-injective if and only if there
exists a pure exact sequence of representations $0 \to X \to T^0
\to T^1\to T^2 \to \cdots$ with $T^i$ strongly fp-injective in $\text{Rep}(Q,R)$ for all $i\geq 0$.
\end{prop}
\begin{proof} The necessity is trivial. For the sufficiency, split the sequence
$0 \to X \to T^0 \to T^1 \to \cdots$ as pure short exact sequences $0\to X^i\to T^i \to X^{i+1}\to 0$
with $X^i=\Ker(T^i\to T^{i+1})$ and $i\geq 0$. Since $T^0$ is strongly fp-injective,
It follows by the definition that there is a
pure exact sequence $0\to T^0\to I^0 \to S^{1}\to 0$
with $I^0$ injective and $S^1$ strongly fp-injective. Consider the following push-out
diagram:
$$\xymatrix{
   &  & 0 \ar[d]_{}  & 0 \ar[d]_{ }  &  \\
 0  \ar[r]^{ } & X \ar@{=}[d]_{ } \ar[r]^{ } & T^0 \ar[d]_{}
  \ar[r]^{ } & X^1 \ar[d]_{ } \ar[r]^{ } & 0 \\
 0  \ar[r]^{ } & X \ar[r]^{} & I^0 \ar[d]_{}
  \ar[r]^{ } & U^1 \ar[d]_{ } \ar[r]^{ } & 0 \\
   &   & S^1 \ar[d]_{ }
  \ar@{=}[r]^{ } & S^1 \ar[d]_{ } & \\
   &   & 0 & 0 &   }
$$
We can check readily  that both the sequences $0\to X\to I^0 \to U^1\to 0$ and $0\to X^1\to U^1 \to S^1\to 0$
are pure exact. Now consider the following push-out diagram:
$$\xymatrix{ & 0 \ar[d]_{ }  & 0 \ar[d]_{ }
   &   &   \\
  0  \ar[r]^{ } & X^1 \ar[d]_{ } \ar[r]^{ } & U^1 \ar[d]_{ }
  \ar[r]^{ } & S^1 \ar@{=}[d]_{ } \ar[r]^{ } & 0 \\
  0  \ar[r]^{ } & T^1 \ar[d]_{ } \ar[r]^{ } & J^1 \ar[d]_{ }
  \ar[r]^{ } & S^1\ar[r]^{ } &0 \\
    & X^2 \ar[d]_{ } \ar@{=}[r]^{ } & X^2 \ar[d]_{ }&   &   \\
    & 0                             & 0   &   &    }
$$

Visibly, $J^1$ is strongly fp-injective since the class of
strongly fp-injective representations is closed under extensions by Proposition \ref{closeprop}.
Moreover,  the sequence $0\to T^1\to J^1 \to S^1\to 0$ is
pure exact as a push-out of two pure exact sequences. By applying the functor $Y\otimes_Q-$ for any representation $Y\in\text{Rep}(Q^{op}, R^{op})$, one gets by the Horse-Shoe Lemma that the sequence $0\to U^1\to J^1 \to X^2\to 0$
is pure exact. Since $J^1$ is strongly fp-injective, there exists a pure exact sequence $0\to J^1\to I^1 \to S^2\to 0$ such that
$I^1$ is injective and $S^2$ is strongly  fp-injective.
Again consider the following push-out diagram:
$$\xymatrix{
   &  & 0 \ar[d]_{}  & 0 \ar[d]_{ }  &  \\
 0  \ar[r]^{ } & U^1 \ar@{=}[d]_{ } \ar[r]^{ } & J^1 \ar[d]_{}
  \ar[r]^{ } & X^2 \ar[d]_{ } \ar[r]^{ } & 0 \\
 0  \ar[r]^{ } & U^1 \ar[r]^{} & I^1 \ar[d]_{}
  \ar[r]^{ } & U^2 \ar[d]_{ } \ar[r]^{ } & 0 \\
   &   & S^2 \ar[d]_{ }
  \ar@{=}[r]^{ } & S^2 \ar[d]_{ } & \\
   &   & 0 & 0 &   }
$$ It is easy to see that the sequence $0\to U^1\to I^1 \to U^2\to 0$ is pure exact.
Assembling the sequences $0\to X\to I^0 \to U^1\to 0$ and $0\to U^1\to I^1 \to U^2\to 0$,  we then have a pure
exact sequence $0\to X\to I^0 \to I^1\to U^2\to 0$  with $I^0$ and $I^1$ injective.
Inductively, we will obtain a pure exact
sequence $0\to X\to I^0 \to I^1\to I^2\to \cdots$
 with $I^i$  injective in $\text{Rep}(Q, R)$ for all $i\geq 0$, as desired.
\end{proof}

\begin{rem} If one replaces the strongly fp-injective representations $T^i$ with the
fp-injective ones in Proposition \ref{Stability}, then  $X$ is not necessarily strongly
fp-injective in $\text{Rep}(Q,R)$. In fact, let $X$ be an fp-injective
representation that is not strongly fp-injective, then one has a pure exact
sequence $0 \to X \to X\to 0 \to \cdots$.
\end{rem}


Our next result relies on \cite[Theorem 5.1]{BCE20} proved by Bazzoni, Cortes-Izurdiaga and Estrada that every acyclic complex of injective modules with fp-injective cycles
is contractible. It is evident that for an acyclic complex of injective modules  the cycles are fp-injective if and only if the cycles are strongly fp-injective.

\begin{prop} Let $Q$ be a {locally target-finite and} right rooted quiver, and $X$ a strongly fp-injective representation in $\text{Rep}(Q,R)$. If there is a pure exact sequence
$\cdots\to I^{-2}\to I^{-1}\to I^{0}\to X\to 0$ with  $ I^{i}$ injective for all $i\leq 0$, then $X$ is injective.
\end{prop}
\begin{proof} Since $X$ is strongly fp-injective, one has a pure  exact sequence of representations
$0\to X\to I^{1}\to I^{2}\to I^3\to \cdots$ with $ I^{i}$ injective for all $i\geq 0$, and then an acyclic complex
$\cdots\to I^{-1}\to I^{0}\to I^{1}\to \cdots$ of injective  representations. Thus, we infer from Propositions \ref{closeprop1} and \ref{closed p1} that all its  cycles are  strongly fp-injective.
Now it follows from Theorem \ref{strongly FP}, \cite[Proposition 2.1]{EEGR09} and \cite[Theorem 5.1]{BCE20} that  $X(v)$ and $\Ker(\psi^X_v)$ are injective
$R$-modules, $\psi^X_v$ is an epimorphism for each vertex $v\in Q_0$. As the quiver $Q$ is right rooted, it follows
from \cite[Theorem 4.2]{EEGR09} that $X$ is injective
in $\text{Rep}(Q,R)$.
\end{proof}


\section{ Relative Gorenstein injective objects}
 Let $\mathcal{B}$ be a subcategory of an abelian category $\mathcal{A}$.
 Following \cite{GI22}, we will say that a chain complex $X$ of  objects in $\mathcal{A}$ is $\Hom_{\mathcal{A}}(\mathcal{B}, -)$-acyclic
 if  $\Hom_{\mathcal{A}}(B, X)$ is an acyclic complex of abelian groups for all $B\in\mathcal{B}$.
 If $X$ itself is also exact we will say  that $X$ is an exact $\Hom_{\mathcal{A}}(\mathcal{B}, -)$-acyclic complex.
 We say an object $N$ Gorenstein $\mathcal{B}$-injective if  $N =\Z_0E$ for some exact $\Hom_{\mathcal{A}}(\mathcal{B}, -)$-acyclic complex $E$
 of injective objects. In particular, when $\mathcal{B}$ is the subcategory of all injective (respectively, fp-injective) $R$-modules in $R$-Mod,
 then the definition recovers the usual Gorenstein (respectively, Ding) injective $R$-modules, consult \cite{EJ95, Holm04} and \cite{DM08, YLL10} for more details about Gorenstein and Ding injective $R$-modules, respectively.

\begin{nota} In the following we let $\mathcal{J}$ denote the class of all strongly fp-injective $R$-modules,
and $\mathcal{GI}$ denote the class of all Gorenstein $\mathcal{J}$-injective (which we call \emph{Gorenstein strongly fp-injective}) $R$-modules,
and we set $\mathcal{W}={^\bot\mathcal{GI}},$ where ${^\bot\mathcal{GI}}$ is defined as the class of all $R$-modules $W$ satisfying $\Ext^1_R(W, G)=0$ for any $G\in \mathcal{GI}$. We would like to point out that the subcategory of Gorenstein strongly fp-injective
$R$-modules shares many nice properties of the one of Gorenstein injective $R$-modules, this can be proved by a series of
similar arguments used to show Gorenstein $\mathcal{B}$-injective $R$-modules, one of them is restated as the following without proof,  see \cite[Lemma 16]{GI22} for details.
\end{nota}

\begin{lem}\label{close products}The subcategory $\mathcal{GI}$ of all Gorenstein strongly fp-injective $R$-modules is closed under direct products and direct summands.
\end{lem}

It is clearly seen that any Gorenstein strongly fp-injective $R$-module is Gorenstein injective,
but the converse is not true in general. By \cite[Theorem 3.5.1]{EHS13}, it follows that a representation
$X$ is Gorenstein injective in $\text{Rep}(Q,R)$ if and only if $X\in \Psi(\mathcal{G})$ whenever $Q$ is right rooted,
where  $\mathcal{G}$ denotes the subcategory of all Gorenstein injective $R$-modules. We show that the similar characterization is possessed by Gorenstein strongly fp-injective representations.

\begin{df} We say a representation $X$ in $\text{Rep}(Q, R)$  Gorenstein strongly fp-injective (respectively,  Gorenstein injective,  Ding injective) if there is an acyclic complex $$\xymatrix{\cdots\ar[r]^{}& I_2 \ar[r]^{f_2}  & I_1 \ar[r]^{f_1} &I_0 \ar[r]^{f_0} & I_{-1} \ar[r]^{f_{-1}} & I_{-2} \ar[r]& \cdots  }$$ of injective representations with $X=\Ker(f_{-1})$, which remains exact after applying the functor $\Hom_{Q}(J,-)$ for any strongly fp-injective (respectively,  injective,  fp-injective) representation $J$.
\end{df}

\begin{lem} \label{cotorsion pair} Let $Q$ be a  right rooted quiver. Then the pair $(\text{Rep}(Q,\mathcal{W}),\Psi(\mathcal{GI}))$ is a hereditary cotorsion pair with $\text{Rep}(Q,\mathcal{W})$ thick in $\text{Rep}(Q,R)$.
\end{lem}
\begin{proof}
By \cite[Propostion 24]{GI22}, the pair $(\mathcal{W},\mathcal{GI})$ is a hereditary cotorsion pair with $\mathcal{W}$ thick in the category of $R$-modules. Then \cite[Theorem B]{HP19} applies.
\end{proof}

Now we are in a position to give the main result in this section.

\begin{thm} \label{main Gorens} Let $Q$ be a {locally target-finite and} right rooted quiver. Then a representation $X$ in $\text{Rep}(Q,R)$ is Gorenstein strongly fp-injective if and only if $X\in\Psi(\mathcal{GI})$.
\end{thm}
\begin{proof} It follows from Lemmas \ref{G-st-inj1} and \ref{G-st-inj2} below.
\end{proof}

\begin{lem} \label{G-st-inj1} Let $Q$ be a right rooted quiver, and $X$ a representation in $\text{Rep}(Q,R)$.
 If $X\in\Psi(\mathcal{GI})$, then $X$ is Gorenstein strongly fp-injective.
\end{lem}
\begin{proof} Let $X\in\Psi(\mathcal{GI})$. Then $\Ext^{i\geq 1}_Q(K, X)=0$ for all $K\in \text{Rep}(Q,\mathcal{W})$
by Lemma \ref{cotorsion pair}. Note that the subcategory of all  strongly fp-injective $R$-modules is contained in
$\mathcal{W}={^\bot\mathcal{GI}}$, and so every strongly fp-injective  representation belongs to $\text{Rep}(Q,\mathcal{W})$ by
Theorem \ref{strongly FP}(1) (where $Q$ is not necessarily locally target-finite). Particularly, $\Ext^{i\geq 1}_Q(J, X)=0$ for  any  strongly fp-injective  representation $J$.
Since any Gorenstein strongly fp-injective $R$-module is always Gorenstein injective, it follows
from \cite[Theorem 3.5.1]{EHS13} that $X$ is Gorenstein injective in $\text{Rep}(Q,R)$. Therefore, there exists an exact
sequence of injective representations $$\xymatrix{\mathbb{I}:= \ \cdots\ar[r]^{}& I_2 \ar[r]^{f_2}  & I_1 \ar[r]^{f_1} &I_0 \ar[r]^{f_0} & X \ar[r]& 0  }$$
 such that $X_{j+1}:=\Ker(f_j)$ are Gorenstein injective for all $j\geq 0$.
Let $J$ be any  strongly fp-injective  representation.  By \cite[Lemma 17]{GI22}, it suffices to show that $\mathbb{I}$ is $\Hom_Q(J, -)$-exact. To this end, decompose $\mathbb{I}$ into short exact sequences
$$\xymatrix{\mathbb{I}_j:= \ 0\ar[r]^{}& X_{j+1}\ar[r]^{}& I_j\ar[r]^{}& X_j\ar[r]^{}& 0}$$
in $\text{Rep}(Q,R)$ in which $X_0=X$. Write  a short exact sequence in $\text{Rep}(Q,R)$
$$\xymatrix{ 0\ar[r]^{}& J\ar[r]^{}& E^0\ar[r]^{}& J^1\ar[r]^{}& 0}$$
with $E^0$ injective and $J^1$ strongly fp-injective. Using dimension shifting, one has  $$\Ext^{1}_Q(J, X_1)\cong\Ext^{2}_Q(J^1, X_1)\cong\Ext^{1}_Q(J^1, X)=0.$$ This implies that the sequence $\mathbb{I}_0$  is $\Hom_Q(J, -)$-exact. Moreover, by dimension shifting, we infer that
$\Ext^{i}_Q(J, X_1)=0$ for all $i\geq 1$. Thus, we have shown that $X_1$ has the same property as $X$. Proceeding in the same manner, we can prove
that each $\mathbb{I}_j$  is $\Hom_Q(J, -)$-exact, and then so is $\mathbb{I}$, as desired.
\end{proof}

\begin{lem} \label{G-st-inj2} Let $Q$ be a {locally target-finite and} right rooted quiver.  If $X$ is a Gorenstein strongly fp-injective representation in $\text{Rep}(Q,R)$,
then $X\in\Psi(\mathcal{GI})$.
\end{lem}
\begin{proof}
By definition, there is an epimorphism $f: I\rightarrow X$ with $I$ an injective representation. This yields the following commutative diagram for any $v\in Q_0$:
 $$\xymatrix{
   I(v) \ar[d]_{\psi^I_v} \ar[r]^{f(v)} & X(v) \ar[d]_{\psi^X_v} \ar[r]^{} &0  \\
   \prod_{a\in Q^{v\rightarrow\ast}_1}I(t(a)) \ar[r]^{} & \prod_{a\in Q^{v\rightarrow\ast}_1}X(t(a)) \ar[r]^{} & 0.  }
 $$
Since $I$ is injective, it follows that $\psi^I_v$ is a splitting  epimorphism,  and hence $\psi^X_v$ is an epimorphism.

Now assume that the acyclic complex of injective representations
$$\mathbb{I}:= \ \xymatrix{\cdots\ar[r]^{}&  I_{1} \ar[r]^{f_{1}} &I_0 \ar[r]^{f_0} & I_{-1} \ar[r]^{f_{-1}}& \cdots  }$$
with $X=\Ker(f_{0})=\Z_0\mathbb{I}$  remains exact after applying $\Hom_{Q}(J,-)$ for any strongly fp-injective representation $J$.
Given a strongly fp-injective $R$-module $F$, since the class of all  strongly fp-injective representations in $\text{Rep}(Q,R)$ is closed under direct
products, see Corollary \ref{strongly ccor}, it follows from Proposition \ref{sub-quiver fp} that
$e_{\rho}^i(F)$ and $ e_{\rho}^i(F)/s_i(F)=\prod_{a\in Q_1^{*\to i}}e_{\rho}^{(s(a))}(F)$ are strongly fp-injective representations
 for any vertex $i\in Q_0$, and so  both the complexes $\Hom_Q( e_{\rho}^i(F), \mathbb{I})$ and  $\Hom_Q( e_{\rho}^i(F)/s_i(F), \mathbb{I})$ are acyclic. Thus the exactness of the sequence
$$\xymatrix{0\ar[r]^{} & s_i(F) \ar[r]^{} &e_{\rho}^i(F) \ar[r]^{} & e_{\rho}^i(F)/s_i(F) \ar[r]^{} & 0  }$$
implies that the complex $\Hom_Q(s_i(F), \mathbb{I})$ is acyclic. As we have proved above,  $\psi^X_i$ is  epimorphic, it follows from \cite[Proposition 3.10]{O19} that $$\Ext^1_R(F, \Ker(\psi^X_i))\cong \Ext^1_Q(s_i(F), X)=0.$$ Similarly, one can show
that  $\psi^{{\Z}_m\mathbb{I}}_i$ is epimorphic and $$\Ext^1_R(F, \Ker(\psi^{\Z_m\mathbb{I}}_i))\cong \Ext^1_Q(s_i(F), {\Z_m\mathbb{I}})=0$$ for any integer $m\in \mathbb{Z}$, where $\Z_m\mathbb{I}=\Ker(f_m)$. Now we infer readily that $\Ker(\psi^X_i)=\Z_0\mathbb{K}(i)=\Ker(h_0)$ is a Gorenstein strongly fp-injective $R$-mdoule, where
$$\mathbb{K}(i):= \ \xymatrix{\cdots\ar[r]^{}&  \Ker(\psi^{I_{1}}_i) \ar[r]^{h_1} &\Ker(\psi^{I_{0}}_i) \ar[r]^{h_0} & \Ker(\psi^{I_{-1}}_i) \ar[r]^{\hspace{4mm}h_{-1}}& \cdots  },$$
see the diagram (4.1) below in which one should take $j=i$.

In the next, we will prove that  $X(i)$ is a Gorenstein strongly fp-injective $R$-module for any vertex $i\in Q_0$.
As $Q$ is right rooted, there  exists some ordinal $\lambda$ such that $Q_0=W_\lambda=\bigcup_{\alpha\leq\lambda}W_\alpha$, see \cite[2.9]{HP19}.
We will show the result by transfinite induction on $\lambda$.
The induction start for $\lambda=0$ is nothing to prove since $W_0$ is empty. For $\lambda=1$, we have $Q$ discrete, that is, there are no arrows in $Q_1$. Thus, the complex $\Hom_R(F, \mathbb{I}(i))\cong\Hom_Q(s_i(F), \mathbb{I})$ is acyclic. Hence $X(i)$ is
a Gorenstein strongly fp-injective $R$-module since $X(i)=\Ker(f_{0}(i))=\Z_0\mathbb{I}(i)$.

Now we assume that $\lambda>1$ is a limit ordinal and  that
$X(i)$ is  Gorenstein strongly fp-injective for each vertex $i\in W_\delta$ and all $\delta<\lambda$. We will prove that  $X(j)$ is  Gorenstein strongly fp-injective for any given $j\in W_\lambda$. However, this is clear since we have $W_\lambda=\bigcup_{\delta<\lambda}W_\delta$ by our assumption that $\lambda$ is a limit ordinal.

It remains to consider the situation where $\lambda=\delta+1$ is a successor ordinal. We assume that
$X(i)$ is a Gorenstein strongly fp-injective $R$-module for each $i\in W_\delta\subseteq W_{\delta+1}$, and we must show that $X(j)$ is  Gorenstein strongly fp-injective for any $j\in W_{\delta+1}$.
Without loss of generality, we may assume that $j\in W_{\delta+1}\setminus \bigcup_{\alpha\leq \delta}W_\alpha$.
Combining
$$W_{\delta+1}=\left \{j\in Q_0 \hspace{0.2cm}
\begin{array}{ |l} j \text{ is not the source of any arrow } \\
 a \text{ in }  Q \text{\ with \ } t(a)\not\in\bigcup_{\alpha\leq \delta}W_\alpha.\end{array}\right\},
$$
we consider the  exact sequence $$\xymatrix{
0 \ar[r]^{} &\Ker(\psi^{X}_j) \ar[r]^{}& X(j) \ar[r]^{\hspace{-10mm}\psi^{X}_j} &\prod_{a\in Q^{j\rightarrow\ast}_1}X(t(a))\ar[r]^{}& 0.}$$
Since $j\in W_{\delta+1}\setminus \bigcup_{\alpha\leq \delta}W_\alpha$, it follows that  $t(a)\in\bigcup_{\alpha\leq \delta}W_\alpha$ for all arrows $a\in Q^{j\rightarrow\ast}_1$, and so all
$X(t(a))$ are Gorenstein strongly fp-injective  $R$-modules by the assumption. Moreover, by Lemma \ref{close products}, one has $\prod_{a\in Q^{j\rightarrow\ast}_1}X(t(a))$  Gorenstein strongly fp-injective. Combining the Gorenstein strongly fp-injectivity of $\Ker(\psi^{X}_j)$ which has been shown before, we obtain that both the acyclic complexes $\mathbb{K}(j)$ and $\prod\mathbb{I}(t(a))$ of injective $R$-modules,
in the following commutative diagram, are $\Hom_R(J, -)$-exact
for any strongly fp-injective $R$-module $J$.
$$\xymatrix@C=3mm@R=8mm{    & 0 \ar[d]_{}             &&  0\ar[d]_{}           & &0 \ar[d]^{}&  & \\
\mathbb{K}(j):= \ \cdots \ar[r]^{} & \Ker(\psi^{I^{-1}}_j) \ar[d]_{} \ar[rr]^{} && \Ker(\psi^{I^{0}}_j) \ar[d]_{} \ar[rr]^{} & &\Ker(\psi^{I^{1}}_j) \ar[r]\ar[d]^{}&\cdots & \\
\mathbb{I}(j):= \ \cdots \ar[r]^{} & I^{-1}(j)\ar[d]_{\psi^{I^{-1}}_j} \ar[rr]^{f^{-1}(j)} && I^{0}(j) \ar[d]_{\psi^{I^{0}}_j} \ar[rr]^{f^{0}(j)} & &I^{1}(j) \ar[r]\ar[d]_{\psi^{I^{1}}_j}&\cdots & {(4.1)}    \\
\prod\mathbb{I}(t(a)):= \ \cdots \ar[r]^{} & \displaystyle\prod_{a\in Q^{j\rightarrow\ast}_1}\hspace{-3mm}I^{-1}(t(a)) \ar[d]_{} \ar[rr]^{} &&
  \displaystyle\prod_{a\in Q^{j\rightarrow\ast}_1}\hspace{-3mm}I^{0}(t(a)) \ar[d]_{} \ar[rr]^{} & &\displaystyle\prod_{a\in Q^{j\rightarrow\ast}_1}\hspace{-3mm}I^{1}(t(a)) \ar[r]\ar[d] &\cdots & \\
     & 0   & &0   & &0 &  }$$
By applying the functor $\Hom_R(J,-)$ with $J$ strongly fp-injective, we then get an exact sequence of complexes
$$\xymatrix{
0 \ar[r]^{} &\Hom_R(J, \mathbb{K}(j)) \ar[r]^{}& \Hom_R(J, \mathbb{I}(j)) \ar[r]^{} &\Hom_R(J,\prod\mathbb{I}(t(a)))\ar[r]^{}& 0.}$$
Now it is easily seen that  the complex $\Hom_R(J, \mathbb{I}(j))$ is acyclic. This implies that
$X(j)$ is a Gorenstein strongly fp-injective $R$-module. Hence, by induction, we can show that $X(i)$ is a Gorenstein strongly fp-injective  $R$-module
for every $i\in Q_0$. In other words, $X\in\Psi(\mathcal{GI})$.
\end{proof}

Dalezios shows in \cite[Proposition 4.6 (2)]{D19} that a representation $X$ in $\text{Rep}(Q, R)$ of a right rooted  quiver $Q$  is Ding injective  if and only if  the canonical homomorphism $\psi_i^X$ is an epimorphism, and the left $R$-modules $\Ker(\psi_i^X)$ and $X(i)$ are Ding injective for each vertex  $i\in Q_0$, provided that $R$ is left coherent. We would like to point out that  Theorem \ref{main Gorens}
reveals incidentally the characterization of Ding injective  representations of $R$-modules when  $R$ is restricting to a coherent ring as the following.

\begin{cor} \label{Ding injective} Let $Q$ be a {locally target-finite and} right rooted quiver, and $R$  a left coherent ring.
Then a  representation $X$ in $\text{Rep}(Q, R)$ is Ding injective  if and only if it is Gorenstein  strongly fp-injective, that is, $X\in \Psi(\mathcal{GI})$
\end{cor}
\begin{proof} Since $R$ is left coherent, it follows  by \cite[Theorem 4.2]{LGO} that an $R$-module is fp-injective if and only if it is strongly fp-injective, and hence by Lemma \ref{proflatfp-injective} (3) and Theorem \ref{strongly FP right}, one gets that a representation is fp-injective  if and only if it is  strongly fp-injective. Thus, the subcategory $\Psi(\mathcal{GI})$ is exactly the one of Ding injective representations, see
  \cite[Proposition 4.6]{D19}.
\end{proof}

Note that if the ring $R$ is left coherent, then absolutely clean left $R$-modules  and fp-injective left $R$-modules coincide. We end the paper by the next result, which improves the model structure in \cite[Proposition 5.10]{BGH}  from the category $R$-Mod of $R$-modules to the category
$\text{Rep}(Q, R)$ of representations of $R$-modules when $R$ is a left coherent ring. It  also improves \cite[Theorem 4.9]{D19} from a Ding-Chen ring to a left coherent ring.

\begin{cor} \label{Ding injective Model} Let $Q$ be a {locally target-finite and} right rooted quiver,  $R$  a left coherent ring,  and $\mathcal{GI}$ be the class of all Ding injective  $R$-modules. Then
there is a hereditary Hovey triple $(\text{Rep}(Q, R), \text{Rep}(Q,\mathcal{W}), \Psi(\mathcal{GI}))$ on the
category $\text{Rep}(Q,R)$ of quiver representations of left $R$-modules. The homotopy
category of this model structure $$\text{Ho}(\text{Rep}(Q,R))\cong\underline{\Psi(\mathcal{GI})},$$ is the stable category of Ding injective representations.
\end{cor}
\begin{proof}
Since $R$ is a left coherent ring, it follows from \cite[Proposition 5.10]{BGH} that the cotorsion pair $(\mathcal{W},\mathcal{GI})$ is
 set-cogenerated, and so it is complete. Moreover, one infers from \cite[Proposition 23]{GI22} that the pair $(\mathcal{W},\mathcal{GI})$ is a hereditary and perfect cotorsion pair with $\mathcal{W}$ thick in the category of $R$-modules. Thus it  follows from \cite[Theorem B]{HP19} and \cite[Theorem 4.6]{O19} that the pair $(\text{Rep}(Q,\mathcal{W}),\Psi(\mathcal{GI}))$
 is a hereditary and perfect cotorsion pair with $\text{Rep}(Q,\mathcal{W})$ thick in $\text{Rep}(Q,R)$.
\end{proof}

\section*{Acknowledgements}
The authors thank the referee for his/her carefully reading and considerable suggestions,
which have improved the present paper. The authors thank professors Li Liang and Zhenxing Di for helpful
discussions related to this work. The work is partly supported by NSF of China (Grant Nos. 12161049; 12361008), and Funds for Innovative Fundamental Research Group Project of Gansu Province (Grant No. 23JRRA684).

\vskip 18pt

{\footnotesize \noindent Qihui Li,\\
 Department of Mathematics, Lanzhou Jiaotong University, Lanzhou 730070, P.R. China\\
 E-mail: 1319241167@qq.com}

\vskip 15pt

 {\footnotesize \noindent Junpeng Wang,\\
 Department of Mathematics,  Northwest Normal University,  Lanzhou {\rm 730070}, P.
R. China\\ E-mail: wangjunpeng1218@163.com}

\vskip 15pt
 {\footnotesize \noindent Gang Yang,\\
1. Department of Mathematics, Lanzhou Jiaotong University, Lanzhou 730070, P.R. China\\
2. Gansu Provincial Research Center for Basic Disciplines of Mathematics and Statistics,
Lanzhou 730070, P. R. China\\
 E-mail: yanggang$\symbol{64}$mail.lzjtu.cn}

\end{document}